%% file: fractionalArcsine-arXiv02.tex
\documentclass[a4paper,12pt]{article}
\input{style-common}

\input{style-e}

\input{style-layout-paper}

\usepackage[sort]{natbib}
\setcitestyle{authoryear,comma,open={(},close={)}} 

\usepackage{hyperref}
\usepackage{mathtools}
\mathtoolsset{showonlyrefs}

\DeclareMathOperator{\disc}{disc}

\begin{document}
\begin{center}
	{\Large \textbf{
  Excursion-dependent inverse-subordinator time changes of regenerative processes: Excursion structure and multiscale occupation-time limits
  }
	}
\end{center}
\begin{center}
	Kosuke Yamato (The University of Osaka)
\end{center}

\begin{abstract}
We study an excursion-dependent generalization of inverse-subordinator time changes for recurrent regenerative processes.
Subordinators from an independent family are assigned to measurable classes of excursions and to the distinguished state at which the process is regenerative, and the random clock is constructed from the corresponding occupation times.
Although the time-changed process is not necessarily regenerative or Markovian, we show that its transformed excursion point process is again Poisson and is obtained from the original excursion point process by a
marking-and-mapping procedure.
We also identify the transformed stagnancy component, on the original local-time scale, as a subordinator independent of this point process.

We use this representation to study occupation- and waiting-time asymptotics.
Under regular variation assumptions on the tails of the L\'evy measures of the cumulative excursion lifetimes and stagnancy components, we prove a
multiscale joint occupation-time limit theorem, including generalized
arcsine laws and Darling--Kac-type limits.
We further obtain joint scaling limits for the last and next visit
times to the distinguished state.
\end{abstract}


\section{Introduction}

Time changes by inverse subordinators have been studied extensively as
probabilistic representations of time-fractional evolution equations.
Let \(L\) be the generator of a Markov process \((X,\bP_x)\), and consider the time-fractional Cauchy problem
\[
\partial_{t}^{\alpha} u = Lu, \quad u(0,x) = f(x),
\]
where \(\alpha \in (0,1)\) and \(\partial_{t}^{\alpha}\) is the Caputo derivative of order \(\alpha\).
Its solution is represented, under suitable conditions, by
\[
u(t,x)=\bE_{x}[f(X_{A_{t}})],
\]
where
\[
A_{t} = \inf\{s > 0 \mid T_{s} > t\}
\]
is the right-continuous inverse of an \(\alpha\)-stable subordinator \(T\) independent of \(X\) such that \(\bE[e^{-qT_{t}}] = e^{-tq^{\alpha}}\); see, for example, \cite{BaeumerMeerschaert2001}.
Thus, replacing the ordinary time derivative by the Caputo derivative corresponds probabilistically to time-changing \(X\) by an inverse stable subordinator.

More general subordinators give rise to generalized fractional derivatives.
Let \(T\) be a subordinator independent of \(X\), with drift coefficient \(\kappa\) and L\'evy measure \(\mu\), and set
\[
w(t):=\mu(t,\infty).
\]
Define the convolution-type derivative
\[
\partial_t^w g(t)
:=
\frac{d}{dt}
\int_0^t
w(t-s)\bigl(g(s)-g(0)\bigr) ds.
\]
The time derivative associated with \(T\) is then \(\kappa\partial_t+\partial_t^w\).
Under suitable conditions, the solution of the corresponding generalized Cauchy problem
\[
(\kappa\partial_t+\partial_t^w)u(t,x)
=
Lu(t,x),
\qquad
u(0,x)=f(x),
\]
is represented by
\[
u(t,x)=\bE_x[f(X_{A_t})],
\]
where \(A\) is the right-continuous inverse of \(T\).
Thus, the correspondence between time-fractional evolution equations and inverse-subordinator time changes extends beyond the stable case to convolution-type time derivatives associated with general subordinators; see \cite{MeerschaertSikorskii2019,Chen2017,Toaldo2015} for details.

Inverse-subordinator time changes also arise naturally in scaling limits.
In particular, for null-recurrent Markov processes, suitably rescaled martingales may converge to Brownian motion time-changed by an independent Mittag--Leffler process, that is, by the inverse of a stable subordinator; see \cite{MR1949295}, and \cite{MR5018513} in the setting of Markov additive processes.

As described above, inverse-subordinator time changes admit a direct analytic description through the corresponding modification of the time derivative.
In contrast, the Markov and regenerative properties of the original process are not necessarily preserved by the time change.
The point of departure of the present paper is that a tractable structure nevertheless survives at the level of excursions.

To explain the idea, we suppose that the subordinator has strictly increasing paths, so that its inverse is continuous.
This assumption is standard in the study of inverse-subordinator time changes, and we adopt it in the present paper as well.
Then each excursion interval of \(X\) corresponds to a unique excursion interval of the time-changed process.
More precisely, an excursion of \(X\) is not split or merged by the time change; only its time parametrization is modified by the inverse of the
increment path of the subordinator over that interval.
Thus, the time change acts excursion by excursion.

This observation is naturally expressed in the language of Poisson point processes.
As is classically known, for the recurrent regenerative process \(X\) considered here, the excursions indexed by local time form a Poisson point process.
The excursion-wise transformation can be viewed as attaching the corresponding subordinator increment paths as marks, conditionally independently given this point process, and then applying a measurable mapping.
The marking-and-mapping theorem for Poisson point processes therefore provides the basic mechanism behind the transformed excursion structure.  One of the main structural results of this paper is to make this argument precise and identify the excursion point process of the time-changed process.

This viewpoint also suggests a natural generalization.
Since the construction is excursion-wise, one may divide the excursion space into measurable classes and assign possibly different subordinators to the classes and to the distinguished state \(o\) at which the process is regenerative.
The resulting clock is excursion-dependent in the sense that the time change
applied during an excursion depends on the class of that excursion.
The present paper develops this construction and shows that the Poisson property is preserved under the time change.

A closely related work is \cite{ChoiClancy}, which developed an excursion theory for reflected Brownian motion time-changed by an inverse subordinator, including a marking-and-mapping description of the transformed excursion measure.
In the present paper, we formulate the transformation directly for the excursion point process of a general recurrent regenerative process.
This formulation also allows the driving subordinator to depend on the measurable class of the excursion and separates the transformed stagnancy component from the excursion lifetimes.


Once the transformed excursion structure is established, we apply it to study the occupation times of the excursion classes, together with the time spent at \(o\).
For finite partitions of the excursion classes preserved by the time change, we derive a multiscale joint limit theorem for occupation times and the time-changed local time under regular variation assumptions on the tails of the L\'evy measure of these cumulative lifetime processes.
According to their tail behavior, we classify the occupation times of the excursion classes and the origin into three asymptotic regimes: dominant, sub-dominant, and residual.
The components with the heaviest tails form the dominant part; their mean occupation times have generalized arcsine laws in the limit. 
Sub-dominant components are negligible on this scale but become visible on smaller, component-dependent scales, where they give rise to Darling--Kac-type limits.
The residual components vanish on the dominant scale.
We also derive the scaling limits of the last and next visiting time to a distinguished state \(o\).
In the proof, Williams' formula plays an important role since it represents the occupation times in terms of the cumulative lifetime processes, which enables us to derive the occupation-time asymptotics to be derived from their scaling limits.

Our occupation-time limit theorem belongs to the line of results initiated by \cite{Lamperti1958}.
He considered stochastic processes whose state space is divided into two parts, with transitions between them occurring through a distinguished recurrent state, and characterized the possible limiting distributions of the mean occupation time in one part.
The resulting family of distributions is now known as Lamperti's generalized arcsine laws.
\cite{MR1022918} subsequently obtained multidimensional extensions of these laws as the occupation-time distributions of skew Bessel processes on multirays.

An approach particularly relevant to the present work was developed by \cite{Watanabe1995} for one-dimensional diffusions.
Using Williams' formula, he related occupation times to cumulative excursion lifetimes and reduced the asymptotic behavior of the occupation times to that of the corresponding subordinators.
\cite{MR3647066} extended this approach to joint occupation times of diffusions on multirays.

Functional versions of this approach have also been developed.
\cite{FujiharaKawamuraYano2007} established functional limit theorems for occupation times of Lamperti's stochastic processes.
\cite{Sera2020} studied occupation-time limits for intermittent maps and obtained a functional joint limit theorem simultaneously extending Darling--Kac-type limits, Lamperti-type generalized arcsine laws, and Dynkin--Lamperti-type limits.
Our formulation and proof of the occupation-time limit theorem are particularly influenced by \cite{Sera2020}.

We summarize the main contributions of this paper as follows.
First, we establish the Poisson property of the transformed excursion point process for excursion-dependent time changes of recurrent regenerative processes and identify the independent transformed stagnancy component on the original local-time scale.
Second, under regular variation assumptions on the tails of the L\'evy measures of components, we prove a multiscale joint occupation-time limit theorem, including generalized arcsine laws and Darling--Kac-type limits, and adjoin functional limits for the last and next visit times.
These limits are obtained jointly with the time-changed local time.

In Section \ref{section:example}, we show how the regular variation assumptions can be verified and illustrate the resulting limits by two classes of examples:
continuous-time Markov chains on discrete rays and one-dimensional diffusions.
In the former case, the excursion lifetime tails reduce to hitting time tails on the individual rays; in the latter, their asymptotics can be obtained from the growth of the cumulative speed measures on the two half-lines.

The paper is organized as follows.
In Section \ref{section:excursionTheoryBackground}, we recall the excursion theory for regenerative processes and introduce measurable partitions of the excursion space.
In Section \ref{section:timeChange}, we construct the excursion-dependent inverse-subordinator time change and prove the main structural result, Theorem \ref{Thm:transformed-excursion-PPP}, which identifies the transformed excursion point process as a Poisson point process.
Section \ref{section:cumulativeOccupationTimeBasics} identifies the transformed stagnancy component and its independence from the transformed excursion point process, establishes the occupation-time identities, and derives the subordination relation for the lifetime exponents.
Section \ref{section:occupationTimeLimitTheorem} proves the multiscale occupation-time limit theorem, Theorem \ref{thm:occupation_limit}, and its joint extension to waiting times, Theorem \ref{thm:waiting-times}, after preparing several auxiliary results on \(J_1\)-convergence.
Section \ref{section:example} illustrates these results through continuous-time Markov chains on discrete rays and one-dimensional diffusions.
Appendix \ref{appendix:proofs} contains the proofs of the auxiliary results on inverse mappings and \(J_1\)-convergence.
Appendix \ref{appendix:nonMarkov} gives sufficient conditions under which a homogeneous inverse-subordinator time change fails even the time-inhomogeneous Markov property.

\subsubsection*{Acknowledgement}
The author was supported by JSPS KAKENHI grant no. 24K22834.

\subsubsection*{Declaration on the use of generative AI}

This work was developed under the author's direction through discussions with OpenAI's ChatGPT-6.
These discussions contributed to the development and refinement of the
mathematical ideas, statements, and proofs, as well as to the preparation
of the manuscript.
The author has reviewed the final manuscript and takes full responsibility
for its content.

\section{Excursion theory for regenerative processes} \label{section:excursionTheoryBackground}

In this section, we recall the excursion-theoretic framework used throughout the paper.
We first summarize the excursion theory for a regenerative process.
We then introduce measurable partitions of the excursion space, which will later be used to assign different random clocks to different classes of excursions.

\subsection{Regenerative processes and excursion point processes}

We recall the basic excursion-theoretic representation of regenerative
processes.
The material in this subsection is standard and is taken mainly
from \cite[Chapter 29]{Kallenberg-third}.

Let \((\Omega,\cA,\bP)\) be a probability space, let \(S\) be a Polish space with a distinguished state \(o\), and let \(D\) denote the space of \(S\)-valued c\`adl\`ag paths on \([0,\infty)\).
Consider an \(S\)-valued c\`adl\`ag process \(X=(X_t)_{t\ge0}\), adapted to a right-continuous complete filtration \(\cF=(\cF_t)_{t\ge0}\), and assume that \(\bP[X_0=o]=1\).
For any c\`adl\`ag path \(x\), we use the convention \(x(0-):=x(0)\).
We write \(\theta\) for the shift operator on \(D\), that is,
\[
(\theta_t x)(s)=x(t+s)
\quad (x\in D,\ s,t\ge0).
\]

We assume that \(X\) is \textit{regenerative} at \(o\) in the following sense: there exists a probability measure \(P_o\) on \(D\) such that, for every \(\cF\)-optional time \(\tau\),
\[
\cL(\theta_{\tau}X \mid \cF_{\tau}) = P_o
\quad
\text{a.s. on } \{\tau<\infty,\ X_{\tau}=o\},
\]
where \(\cL(\theta_\tau X\mid \cF_\tau)\) denotes a regular conditional distribution of \(\theta_\tau X\) given \(\cF_\tau\).
Define
\[
\Xi:=\{t\ge0 \mid X_t=o\}.
\]
Throughout the paper, we assume that the following hold almost surely:
\begin{enumerate}
  \item The closure \(\overline{\Xi}\) is perfect, that is, \(\overline{\Xi}\) has no isolated points.
  \item The set \(\Xi\) is not locally finite, that is, for some \(n\in\bN\), the set \(\Xi\cap[0,n]\) is infinite.
  \item \(\Xi \neq [0,\infty)\).
\end{enumerate}
These assumptions merely exclude certain specific cases.
In fact, the possible forms of \(\Xi\) are highly restricted; see \cite[Theorem 29.7]{Kallenberg-third}.

\begin{Rem}
The assumptions are not very restrictive in the situations of interest.
For example, they are satisfied when \((X_t,\bP_{x})\) is a standard process in the sense of \cite{BluGet1968} and \(o\) is regular for itself, namely
\[
\bP_o(T_o=0)=1
\quad \text{for} \quad
T_o:=\inf\{t>0 \mid X_t=o\},
\]
and not absorbing.
\end{Rem}

By the right-continuity of the paths of \(X\), every point of \(\overline{\Xi}\setminus \Xi\) is isolated from the right.
Since \([0,\infty)\setminus \overline{\Xi}\) is an open subset of \((0,\infty)\), it is a disjoint union of at most countably many open intervals. It follows that
\[
[0,\infty)\setminus \Xi=\bigsqcup_{n=1}^\infty I_n,
\]
where each \(I_n\) is of the form \((a,b)\) or \([a,b)\), with \(0\le a<b\le\infty\).
We call such intervals the \textit{excursion intervals} of \(X\) away from \(o\).

The space of excursions \(D_{0}\) is defined by
\[
D_0
:=
\{e\in D \mid \zeta(e) > 0 \text{ and } e_t=o\ \text{for all } t\ge \zeta(e)\},
\]
where the \textit{lifetime} \(\zeta(e)\) is defined by
\[
\zeta(e):=\inf\{t>0 \mid e_t=o\},
\]
with the convention \(\inf\emptyset=\infty\).

In this setting, the excursion decomposition of \(X\) away from \(o\) is described by a Poisson point process on \([0,\infty)\times D_0\).
The following is a fundamental result of the excursion theory for regenerative processes.

\begin{Thm}[{\citet[Theorem 29.11, 29.13]{Kallenberg-third}}] \label{thm:existenceExcursionMeasure}
There exist a measure \(n\) on \(D_0\), a Poisson point process \(N\) on \([0,\infty)\times D_0\) with intensity measure \(dt\otimes n\), a non-negative non-decreasing continuous adapted process \(L=(L_t)_{t\ge0}\) with \(L_{0} = 0\), and a constant \(c\ge0\) such that the following hold:
\begin{enumerate}
\item We have
\[
\int_{D_{0}}(1 \wedge \zeta(e))n(de) < \infty.
\]
In particular, for every \(t > 0\) and \(h > 0\)
\[
N([0,t]\times\{\zeta>h\})<\infty
\quad\text{a.s.}
\]
\item
Define the $[0,\infty]$-valued process $\eta=(\eta_s)_{s\ge0}$ by
\[
\eta_s
=
cs+\int_{[0,s]\times D_0}\zeta(e)\,N(du\,de)
\quad (s\ge0).
\]
Then $\eta$ is non-decreasing and right-continuous, and almost surely
\[
\eta_s=\inf\{t\ge0:L_t>s\}
\quad (s\ge0), \quad 
L_t=\inf\{s\ge0:\eta_s>t\}
\quad (t\ge0),
\]
and
\[
L_\infty=\inf\{s\ge0:\eta_s=\infty\}.
\]
Moreover, $\eta$ is strictly increasing on $[0,L_\infty)$.
\item
The behavior of \(X\) away from \(o\) is reconstructed from \(N\) and \(\eta\) as follows: for every \((s,e)\in \Supp N\) with \(s \leq L_\infty\),
\[
X_t=e_{t-\eta_{s-}}
\quad
\text{for } t\in[\eta_{s-},\eta_s),
\]
where \(\eta_{s-}:=\lim_{r\uparrow s}\eta_r\).
Equivalently,
\[
e_{u} = X_{\eta_{s-}+u}
\quad \text{for } u \in [0,\zeta(e)).
\]
\item
The Stieltjes measure \(dL\) is almost surely supported on \(\overline{\Xi}\).

\item
For every \(t\ge0\),
\[
\lambda(\Xi\cap[0,t])=cL_t
\quad \text{a.s.},
\]
where \(\lambda\) denotes the Lebesgue measure.
\end{enumerate}
\end{Thm}

We call \(L\), \(\eta\), \(N\), \(n\), \(c\) the \textit{local time}, the \textit{inverse local time}, the \textit{excursion point process}, the \textit{excursion measure}, and the \textit{stagnancy rate} at \(o\), respectively.

For the remainder of the paper, we assume that $X$ is recurrent at $o$,
in the sense that
\[
\sup\{t\ge0:X_t=o\}=\infty
\quad\text{a.s.}
\]
Then
\[
n(D_{0} \setminus D_{0,\mathrm{fin}}) = 0 \quad \text{for} \quad D_{0,\mathrm{fin}} := \{e \in D_{0} \mid  \zeta(e) < \infty \}.
\]
Moreover, almost surely,
\[
L_\infty=\infty,
\quad
\eta_s<\infty\quad\text{for every }s\ge0.
\]
In particular, $\eta$ has strictly increasing paths on $[0,\infty)$.

\subsection{Measurable partitions of the excursion space}

In the construction of the time-changed process, different random clocks will be assigned to different classes of excursions.
We therefore introduce a measurable partition of the excursion space.

Let \(I = \{1,2,\dots,N\}\) or \(\{1,2,\dots\}\) and let $(D_0^{(i)})_{i\in I}$ be a family of measurable subsets of $D_0$ satisfying
\begin{align}
n\left(D_0\setminus\bigcup_{i\in I}D_0^{(i)}\right)=0,
\quad
D_0^{(i)}\cap D_0^{(j)}=\emptyset
\quad(i,j\in I,\ i\ne j).
\label{Eq03}
\end{align}
We call such a family a \textit{measurable partition} of $D_0$,
allowing an $n$-null remainder.
Thus, every excursion in $\bigcup_{i\in I}D_0^{(i)}$
belongs to exactly one class, whereas excursions in the
$n$-null remainder are left unclassified.
We interpret \(D_{0}^{(i)}\) as the class of excursions of type \(i\). 

In typical examples, this partition arises from a measurable partition of the state space,
\[
S=\bigsqcup_{i\in I} S_{i} \sqcup \{o\},
\]
with the property that \(n\)-almost every excursion remains entirely within a single component \(S_i\) during its lifetime. In such a case, we can take
\[
D_{0}^{(i)}
=
\{e\in D_{0} \mid e_{t}\in S_{i}\ \text{for all }0<t<\zeta(e)\},
\]
provided that this defines a measurable subset of \(D_{0}\).
We shall occasionally refer to this as the \textit{state-space partition setting}.

By the basic property of Poisson point processes, we obtain the following decomposition of the excursion point process.

\begin{Prop}
For each \(i\in I\), define the restricted point process
\[
N^{(i)} := N(\cdot\cap ([0,\infty)\times D_{0}^{(i)})),
\]
and let
\[
n^{(i)} := n(\cdot \cap D_{0}^{(i)}).
\]
Then \(\{N^{(i)}\}_{i\in I}\) is a family of independent Poisson point processes with intensity measures \(dt\otimes n^{(i)}\). Moreover,
\[
N=\sum_{i\in I} N^{(i)}
\]
as random measures, almost surely.
\end{Prop}

\section{Excursion-dependent inverse-subordinator time changes}
\label{section:timeChange}

In this section, we introduce the excursion-dependent time-changed process \(X^{\ast}\), which generalizes the time change by an inverse subordinator.
The process \(X^{\ast}\) is our main target of consideration in the present paper.

We first define the random clock \(T\) and its right-continuous inverse \(A\), and set \(X^\ast_{t}:=X_{A_{t}}\).
We then describe how each excursion of \(X\) is transformed into an excursion of \(X^\ast\).
The crucial point is that the transformed excursions form a Poisson point process obtained from the original excursion point process by a marking-and-mapping procedure.
While the time change not necessarily preserves the Markov and regenerative properties, this allows us to retain a tractable excursion structure.
This observation underlies the main structural results of the present paper.

\subsection{Construction of the random clock} \label{section:randomClock}

In this subsection, we construct \(A\) and define the time-changed process \(X^{\ast}\).

Let \((D_{0}^{(i)})_{i\in I}\) be a measurable partition of \(D_{0}\) introduced in the previous section.
Let \(J\) be a finite or countable set and let \((Z^{(j)})_{j\in J}\) be independent subordinators that are also independent of \(X\).
We suppose each \(Z^{(j)}\) has almost surely strictly increasing paths.
In other words, when the L\`evy--Khintchine representation of the Laplace exponent \(\psi^{(j)}\) of \(Z^{(j)}\) is given by
\begin{align}
\psi^{(j)}(q) := -\log E[e^{-qZ^{(j)}_{1}}] = a_{j}q + \int_{0}^{\infty}(1 - e^{-qx})\nu_{j}(dx) \quad (q \geq 0), \label{Eq06}
\end{align}
we assume \(a_{j} > 0\) or \(\nu_{j}(0,1) = \infty\) (see e.g., \cite[p.16]{BertoinLevy}).

Write \(I_0:=I\cup\{0\}\), and let
\[
f:I_0\to J
\]
be a given map.
The role of \(f\) is to assign to each excursion class \(D_0^{(i)}\) and the singleton \(\{o\}\) one of the subordinators \((Z^{(j)})_{j\in J}\).

For each \(i\in I\), we define the occupation time of excursions in \(D^{(i)}_{0}\) up to time \(t\) by
\[
O^{(i)}(t) := \int_{[0,L(t))\times D_0} \zeta(e) N^{(i)}(ds\,de) + \left(t-\eta(L(t)-)\right) 1\{N^{(i)}(\{L(t)\}\times D_0)=1\} \quad (t\geq 0).
\]
We also define the occupation time at \(o\) up to time \(t\) by
\[
O^{(0)}(t) := \int_{0}^{t}1\{X_{s} = o\}ds \quad (t\geq 0).
\]
For each \(j \in J\), we group together the occupation times of all excursion classes assigned to the same subordinator \(Z^{(j)}\), that is, we set
\[
\tilde{O}^{(j)}(t)
:=
\sum_{i \in I_{0},f(i)=j} O^{(i)}(t)
\quad (t\geq0).
\]
We then define the random clock by
\[
T_{t}
:=
\sum_{j \in J} Z^{(j)}\left(\tilde{O}^{(j)}(t)\right) 
\quad (t \geq 0).
\]
Note that \(t \mapsto T_{t}\) is strictly increasing for \(t < \inf\{t \geq 0 \mid T_{t} = \infty \} \).
Throughout the present paper, we always assume that
\begin{align}  
T_t<\infty
\quad \text{for every } t\ge0
\quad \text{a.s.} \label{Eq04}
\end{align}
Its right-continuous inverse is defined by
\[
A_u := \inf\{t\ge0 \mid T_t>u\} \quad (u\ge0),
\]
and we finally define the time-changed process
\[
X_u^\ast:=X_{A_u} \quad (u \geq 0).
\]

\begin{Rem}
As follows from Proposition \ref{prop:TetaIsSubordinator} below,
the process \(T_{t}\) satisfies \(T_{t} \to\infty\) as \(t\to\infty\).
Consequently, $A$ is finite-valued and continuous, with $A_0=0$, and $X^\ast$ has c\`adl\`ag paths.
\end{Rem}

\begin{Rem}
If \(J\) is finite, the finiteness \eqref{Eq04} is obviously satisfied.
When \(J\) is infinite, under \eqref{Eq04}, the process \(T=(T_t)_{t\ge0}\) is almost surely non-decreasing and c\`adl\`ag.
In both cases, since
\[
\sum_{i\in I_0} O^{(i)}(t)=t \quad (t \geq 0),
\]
and each \(Z^{(j)}\) has almost surely strictly increasing paths, it follows that \(T\) is almost surely strictly increasing.
Consequently, \(A\) is almost surely continuous.
\end{Rem}

\begin{Rem}
The usual time change by an inverse subordinator corresponds to the case in which \(f\) is constant.
We refer to this as the \textit{homogeneous case}.
Indeed, when \(f \equiv j_{0}\), we see that \(T_{t} = Z^{(j_{0})}(t)\) and \(A_{u}\) is the right-continuous inverse of \(Z^{(j_{0})}\).
Thus, the present construction extends the usual inverse subordinator time change by allowing different excursion classes to be driven by different subordinators.

The original process is also included as the degenerate homogeneous case: if \(Z^{(j_0)}_t=t\), then \(T_t=t\), \(A_t=t\), and hence \(X^\ast=X\).
This observation implies that our framework can also treat the original process \(X\).
\end{Rem}

\begin{Rem}\label{rem:non-preservation}
The time-changed process \(X^\ast\) is not, in general, regenerative or Markovian, even when \(X\) is Markovian.
We illustrate this in the homogeneous case; \(T_t=Z_t\).
We denote the Laplace exponent of \(Z\) by \(\psi\).

Let \(X\) be a non-explosive recurrent continuous-time Markov
chain on a countable state space, started at \(o\), and suppose that
its holding time at \(o\) is exponential with parameter \(q>0\).
Let
\[
\tau:=\inf\{t\ge0:X_t\ne o\},
\quad
\tau^\ast:=\inf\{t\ge0:X^\ast_t\ne o\}.
\]
Since \(\tau^\ast=Z_\tau\) almost surely, independence gives
\[
\bE[e^{-\lambda\tau^\ast}]
=
\bE[e^{-\tau\psi(\lambda)}]
=
\frac{q}{q+\psi(\lambda)}
\quad(\lambda \geq 0).
\]
Unless \(\psi\) is linear, this is not the Laplace transform of an
exponential distribution.
A regenerative process with a positive holding time at its regeneration point must have an exponential holding time; see \cite[Theorem 29.7(iii)]{Kallenberg-third}.
Thus, \(X^\ast\) is not regenerative at \(o\), except when \(Z\) is a
deterministic drift.
By the same reason, the process \(X^{\ast}\) is not time-homogeneous Markovian.

The failure of the Markov property is not restricted to processes with
holding times.
If the L\'evy measure \(\nu\) of \(Z\) satisfies \(\nu((0,\infty))=\infty\) and \(X\) has no non-trivial constant intervals almost surely, then \(X^\ast\) does not even satisfy the time-inhomogeneous Markov property.
We discuss this in detail in Appendix \ref{appendix:nonMarkov}.
\end{Rem}

\begin{Rem}
In the state-space partition setting, we have
\[
O^{(i)}(t)=\int_0^t  1_{\{X_s\in S_i \}}\,ds,
\]
and the above construction reduces to a state-dependent random clock driven by the occupation times of the sets \(S_i\).
\end{Rem}

In the inverse-local-time scale, the clock \(T\) becomes a subordinator.
A more refined decomposition, separating the time spent at \(o\) from the
excursion lifetimes, will be given later in Proposition \ref{prop:Gamma0LawAndIndependence}.

\begin{Prop} \label{prop:TetaIsSubordinator}
The process \((T_{\eta_t})_{t\ge0}\) is a subordinator with Laplace exponent
\[
c \psi^{(f(0))}(q)
+
\sum_{j\in J}\sum_{i:f(i)=j}
\int_0^\infty \bigl(1-e^{-\psi^{(j)}(q)x}\bigr) n^{(i)}(\zeta\in dx).
\]
In particular, \(T_{t} \to \infty\) as \(t \to \infty\) almost surely.
\end{Prop}

\begin{proof}
Note that
\[
O^{(i)}(\eta_t)=\int_{[0,t]\times D_0}\zeta(e) N^{(i)}(ds\,de)
\quad (i\in I)
\]
and, by Theorem \ref{thm:existenceExcursionMeasure} (v),
\[
O^{(0)}(\eta_t)=ct
\quad \text{a.s.}
\]

For each \(j\in J\), define
\[
B_t^{(j)}
:=
ct 1_{\{f(0)=j\}}
+
\sum_{i:f(i)=j}\int_{[0,t]\times D_0}\zeta(e) N^{(i)}(ds\,de).
\]
Then \((B^{(j)})_{j\in J}\) is a family of independent subordinators, and it is independent of the family \((Z^{(j)})_{j\in J}\).
By Bochner's subordination, the processes \((Z^{(j)}(B_t^{(j)}))_{t\ge0}\) are independent subordinators with Laplace exponents
\[
c\,1_{\{f(0)=j\}}\psi^{(j)}(q)
+
\sum_{i:f(i)=j}\int_0^\infty \bigl(1-e^{-\psi^{(j)}(q)x}\bigr) n^{(i)}(\zeta\in dx).
\]
Since
\[
T_{\eta_t}
=
\sum_{j\in J} Z^{(j)}(B_t^{(j)}),
\]
it follows that \((T_{\eta_t})_{t\ge0}\) is a subordinator with the desired Laplace exponent.
\end{proof}

\subsection{The transformed excursion point process} \label{section:excursionTimeChange}

We show that the excursions of the time-changed process \(X^{\ast}\) are governed by a new Poisson point process obtained by transforming the excursions of \(X\) one by one according to the increments of the driving subordinators.

Since \(A\) is continuous and non-decreasing, each excursion interval of \(X\) gives rise to a unique excursion interval of \(X^\ast\), obtained by transforming the original interval through the time change.
The following statements illustrate this correspondence precisely.

\begin{Prop} \label{prop:excursionIntervalCorrespondence}
Let \(I=(g,d)\) or \(I=[g,d)\) be an excursion interval of \(X\). Then
\[
A_u\in(g,d)
\quad\Longleftrightarrow\quad
u\in(T_g,T_{d-}),
\]
and
\[
A_u\in[g,d)
\quad\Longleftrightarrow\quad
u\in[T_{g-},T_{d-}).
\]
Consequently, the corresponding excursion interval of \(X^\ast\) is \((T_g,T_{d-})\) or \([T_{g-},T_{d-})\), respectively.
\end{Prop}

\begin{proof}
Since \(A\) is continuous and non-decreasing and \(T\) is strictly increasing, we may easily check the following equivalences:
\[
A_u>t \iff u>T_t,
\quad
A_u\ge t \iff u\ge T_{t-},
\quad
A_u<t \iff u<T_{t-}.
\]
The stated equivalences follow immediately from these relations.
\end{proof}

\begin{Cor}\label{cor:excursionIntervalCorrespondence}
Conversely, let \(I'=(g',d')\) or \(I'=[g',d')\) be an excursion interval of \(X^{\ast}\).
Then the corresponding excursion interval of \(X\) is \((A_{g'},A_{d'})\) or \([A_{g'},A_{d'})\), respectively.
\end{Cor}

\begin{proof}
Take \(u\in I'\).
Then \(A_{u}\) belongs to some excursion interval \(I\) of \(X\).
By Proposition \ref{prop:excursionIntervalCorrespondence}, the corresponding excursion interval of \(X^\ast\) obtained from \(I\) contains \(u\).
Since \(I'\) is an excursion interval of \(X^\ast\), it must coincide with that transformed interval.
Applying \(A\) to the endpoints yields the claim.
\end{proof}

\begin{Rem} \label{rem:EndpointOfExcursionInterval}
Since the endpoints of excursion intervals form an at most countable set for each sample path of \(X\), and since the driving subordinators are independent of \(X\) and have no fixed jump times, we may easily confirm that, outside a null set, the process \(T\) is continuous at every excursion endpoint of \(X\).
Hence, outside the null set, one may replace \(T_{g-}\) and \(T_{d-}\) by \(T_{g}\) and \(T_d\), respectively.
We shall use this continuity property without further mention.
\end{Rem}

For \(j \in J\), \(a \geq 0\), and \(\ell > 0\), define
\[
M^{(j)}_{a,\ell}(u)
:=
Z^{(j)}\bigl((u+a)\wedge(a+\ell)\bigr)-Z^{(j)}(a)
\quad (u \geq 0).
\]
Thus, \(M^{(j)}_{a,\ell}\) records the increment path of \(Z^{(j)}\) over an interval of length \(\ell\), starting from time \(a\).
Let 
\[
\cM
=
\left\{
m: [0,\infty) \to[0,\infty) \;\middle|\;
\begin{array}{l}
m\text{ is c\`adl\`ag and non-decreasing},\\
m(0)=0,\quad m(t)>0\text{ for every }t>0
\end{array}
\right\}.
\]
For \(m\in\cM\), let
\[
m^{-1}(t):=\inf\{u\ge0 \mid m(u)>t\}
\quad (t\ge0),
\]
be the right-continuous inverse of \(m\).
We then define the map \(\Theta:D_{0,\mathrm{fin}}\times\cM\to D_{0,\mathrm{fin}}\)
by
\[
\Theta(e,m)_t:=e_{m^{-1}(t)}
\quad (t\ge0),
\]
where we use the convention \(e_{\infty}=o\).

The following lemma identifies the transformed excursions of \(X^\ast\) in terms of the original excursion point process \(N\).

\begin{Lem}\label{Lem:excursionwise-representation}
Let $(s,e)\in\Supp N$, and suppose that
$e\in D_0^{(i)}$ and $f(i)=j$.
Then the excursion of $X^\ast$ corresponding to $(s,e)$ is given by
\begin{equation}\label{Eq05}
e^\ast
=
\Theta\left(
e,
M^{(j)}_{\tilde{O}^{(j)}(\eta_{s-}),\zeta(e)}
\right).
\end{equation}
Consequently, the family of transformed excursions
\[
\{e^\ast \mid (s,e)\in\Supp N\}
\]
is conditionally independent given $N$.
Moreover,
\begin{equation}\label{Eq19}
\cL\left(
M^{(j)}_{\tilde{O}^{(j)}(\eta_{s-}),\zeta(e)}
~\middle|~N
\right)
=
\cL\left(
\bigl(Z^{(j)}(u\wedge\zeta(e))\bigr)_{u\ge0}
~\middle|~N
\right)
\quad\text{a.s.}
\end{equation}
In particular, the conditional law of $e^\ast$ given $N$
depends on $N$ only through $e$.
\end{Lem}

\begin{proof}
Let \((s,e)\in\Supp N\), \(e\in D_0^{(i)}\), and \(f(i)=j\).
By the excursion representation of $X$, we have
\[
X_t=e_{t-\eta_{s-}},
\quad t\in[\eta_{s-},\eta_s).
\]
Moreover, by the definition of \(T\) and \(\tilde O^{(j)}\), for \(u\in[0,\eta_s-\eta_{s-})=[0,\zeta(e))\),
\begin{align*}
T(u+\eta_{s-})-T(\eta_{s-})
&=
Z^{(j)} \left(\tilde O^{(j)}(u+\eta_{s-})\right)
-
Z^{(j)} \left(\tilde O^{(j)}(\eta_{s-})\right) \\
&=
Z^{(j)} \left(u+\tilde O^{(j)}(\eta_{s-})\right)
-
Z^{(j)} \left(\tilde O^{(j)}(\eta_{s-})\right) \\
&=
M^{(j)}_{\tilde O^{(j)}(\eta_{s-}),\zeta(e)}(u).
\end{align*}
Hence, for \(t\in[0,T(\eta_s)-T(\eta_{s-}))\),
\begin{align*}
A\bigl(t+T(\eta_{s-})\bigr)-\eta_{s-}
&=
\inf\{u\ge0 \mid T(u+\eta_{s-})>t+T(\eta_{s-})\} \\
&=
\inf\{u\ge0 \mid M^{(j)}_{\tilde O^{(j)}(\eta_{s-}),\zeta(e)}(u)>t\} \\
&=
\left(M^{(j)}_{\tilde O^{(j)}(\eta_{s-}),\zeta(e)}\right)^{-1}(t).
\end{align*}
Therefore,
\[
X^\ast_{t+T(\eta_{s-})}
=
X_{A(t+T(\eta_{s-}))}
=
e_{A(t+T(\eta_{s-}))-\eta_{s-}}
=
e\left({\left(M^{(j)}_{\tilde O^{(j)}(\eta_{s-}),\zeta(e)}\right)^{-1}(t)}\right).
\]
This proves \eqref{Eq05}.
Since different excursions of \(X\) correspond to disjoint increment intervals of the corresponding subordinators, the conditional independence of transformed excursions follows.

Conditionally on $N$, the quantities $\tilde{O}^{(j)}(\eta_{s-})$, $\zeta(e)$, and $j$ are fixed.
By independence of the driving subordinators from $N$, the conditional law of \(M^{(j)}_{\tilde{O}^{(j)}(\eta_{s-}),\zeta(e)}\) is the same as that of
\(Z^{(j)}(\cdot\wedge\zeta(e))\).
This proves \eqref{Eq19}.
\end{proof}

We now define the transformed excursion point process \(N^\ast\) by
\[
N^\ast
:=
\sum_{(s,e) \in \Supp N}\delta_{(s,e^{\ast})}.
\]
That is, \(N^\ast\) is the point process on \([0,\infty)\times D_0\) with unit masses at the transformed excursions \((s,e^\ast)\) corresponding to the atoms \((s,e)\) of \(N\).

\begin{Rem}
The point process \(N^\ast\) is almost surely locally finite.
More precisely, we have for every \(t\ge0\) and \(h>0\),
\begin{equation}\label{Eq02}
N^\ast\bigl([0,t]\times\{\zeta > h\}\bigr)<\infty
\quad \text{a.s.}
\end{equation}
Indeed, for each \((s,e^\ast)\in\Supp N^\ast\), since we have
\[
\zeta(e^\ast)=T(\eta_s)-T(\eta_{s-}),
\]
it follows
\[
\int_{[0,t]\times D_0}\zeta(e)\,N^\ast(ds\,de)
\leq
T(\eta_t)<\infty
\quad \text{a.s.},
\]
which implies \eqref{Eq02}.
\end{Rem}

We next check that \(N^\ast\) is obtained by a marking-and-mapping procedure for \(N\) and is therefore a Poisson point process.

Define a probability kernel \(\nu\) from \(D_0\) to \(D_0\).
For $e\in D_{0,\mathrm{fin}}\cap D_0^{(i)}$, define
\[
\nu(e,\cdot)
:=
\cL\left(
\Theta\left(e,M^{(f(i))}_{0,\zeta(e)}\right)
\right).
\]
For all remaining $e\in D_0$, set $\nu(e,\cdot):=\delta_e$.
The latter choice concerns only an $n$-null set.
Let \(n^\ast\) be the measure on \(D_0\) defined by
\[
n^\ast(de)
:=
\int_{D_0}\nu(e',de)n(de').
\]

\begin{Thm}\label{Thm:transformed-excursion-PPP}
The transformed excursion point process \(N^\ast\) of \(X^\ast\) is a Poisson point process on \([0,\infty)\times D_0\) with intensity measure
\[
ds\otimes n^\ast.
\]
\end{Thm}

\begin{proof}
By Lemma \ref{Lem:excursionwise-representation}, conditionally on $N$,
the transformed excursions are independent, and the excursion
associated with an atom $(s,e)$ has conditional law $\nu(e,\cdot)$.
The conclusion therefore follows from the marking and mapping theorem
for Poisson point processes; see, for example,
\cite[Theorem 15.3]{Kallenberg-third}.
\end{proof}

For each \(i\in I\), define \(N^{\ast,(i)}\) by the transformed point process obtained from \(N^{(i)}\) by the same excursion-wise time-change procedure.
Precisely, we define
\[
N^{\ast,(i)} := \sum_{(s,e) \in \Supp N^{(i)}}\delta_{(s,e^{\ast})},
\]
where \(e^{\ast}\) is defined by \eqref{Eq05}.
Note that, in general, the fact that an excursion \(e\) of \(X\) belongs to \(D_0^{(i)}\) does not imply that the transformed excursion \(e^\ast\) also belongs to \(D_0^{(i)}\).

We say that the partition $(D_0^{(i)})_{i\in I}$ is
\textit{preserved by the time change} if, for every $i\in I$,
\[
\nu(e,D_0^{(i)})=1
\quad\text{for $n$-almost every }e\in D_0^{(i)}.
\]
Under this condition, the partition property \eqref{Eq03} also holds with \(n\) replaced by \(n^{\ast}\), which can be confirmed immediately from the definition of \(n^{\ast}\).
Also note that this condition is automatically satisfied in the state-space partition setting.
The following proposition shows that, under the partition preservation assumption, the transformed excursion measure \(n^\ast\) is concentrated on the same measurable partition \((D_0^{(i)})_{i\in I}\).

\begin{Prop}
Suppose that the measurable partition \((D_0^{(i)})_{i\in I}\) is preserved by the time change.
Then
\[
N^{\ast,(i)} = N^{\ast}\bigl(\cdot \cap([0,\infty)\times D_0^{(i)})\bigr)
\quad \text{a.s.}
\]
Consequently, the family \((N^{\ast,(i)})_{i\in I}\) consists of independent Poisson point processes with intensity measures
\[
ds\otimes n^{\ast,(i)} \quad \text{for} \quad 
n^{\ast,(i)} :=n^\ast(\cdot \cap D_0^{(i)}).
\]
\end{Prop}

\begin{proof}
Almost surely, every atom $(s,e)$ of $N$ belongs to a unique class
$D_0^{(k)}$.
By the preservation assumption, its transformed excursion belongs to
\(D_0^{(k)}\) almost surely.
Since $N$ has at most countably many atoms, these statements hold
simultaneously for all its atoms outside a single null set.
Thus, for every $i\in I$,
\[
e\in D_0^{(i)}
\quad\Longleftrightarrow\quad
e^\ast\in D_0^{(i)}
\]
for all corresponding atoms.
Summing over the atoms proves the asserted identity.

The independence and the intensity measures then follow from the
restriction property of Poisson point processes and the partition
property with respect to $n^\ast$.
\end{proof}

\section{Cumulative excursion lifetimes and occupation times}
 \label{section:cumulativeOccupationTimeBasics}

In this section, we introduce the cumulative excursion lifetimes and occupation times, and present basic properties of them.
We first investigate the Laplace exponents of the cumulative excursion lifetimes and their regular variation.
We then relate these processes to the occupation times of \(X^\ast\) in the specific timescale, with special attention to the occupation time at \(o\).
Throughout this section, we assume that the measurable partition \((D_0^{(i)})_{i\in I}\) is preserved by the time change.

\subsection{Cumulative excursion lifetimes}

Define the \textit{cumulative excursion lifetimes} of \(X^{\ast}\) by
\[
\Gamma_s:=\int_{[0,s]\times D_0}\zeta(e) N^\ast(du\,de),
\quad
\Gamma^{(i)}_s:=\int_{[0,s]\times D_0}\zeta(e) N^{\ast,(i)}(du\,de)
\quad (s\ge0,\ i\in I),
\]
and denote their Laplace exponents by \(\Psi^\ast\) and \(\Psi^{\ast,(i)}\), respectively.
Likewise, we denote by \(\Psi^{(i)}\) the Laplace exponent of the cumulative excursion lifetime associated with \(N^{(i)}\), namely
\[
\Psi^{(i)}(q)
=
\int_0^\infty (1-e^{-qt})\,n^{(i)}(\zeta\in dt)
\quad (q\ge0).
\]

At the level of cumulative excursion lifetimes, the marking procedure acts in the same way as Bochner's subordination.

\begin{Prop} \label{prop:LEofTimeChange}
  Let \(i \in I\) and \(f(i) = j\).
  Then
  \[
  \Psi^{\ast,(i)}(q) = \Psi^{(i)}(\psi^{(j)}(q)) \quad (q \geq 0),
  \]
  where \(\psi^{(j)}\) is the Laplace exponent of \(Z^{(j)}\) defined in \eqref{Eq06}.
\end{Prop}

\begin{proof}
  From the definition of \(n^{\ast}\) and the preservation of the partition, we have 
  \begin{align}
    \Psi^{\ast,(i)}(q) &= \int_{D_{0}}n^{(i)}(de)\int_{0}^{\infty}(1 - e^{-qt})P(Z^{(j)}_{\zeta(e)} \in dt) \\
    &= \int_{0}^{\infty}(1 - e^{-\psi^{(j)}(q)t})n^{(i)}(\zeta \in dt) \\
    &= \Psi^{(i)}(\psi^{(j)}(q)).
  \end{align}
  The proof is complete.
\end{proof}

From this identity for the Laplace exponent \(\Psi^{\ast,(i)}\), we have the following index transform under the regular variation assumptions,
which is useful when considering examples.
Here and hereafter, we adopt the convention in asymptotic notation in \cite[p.xix]{Regularvariation}: we interpret \(f(x) \sim cg(x)\) for \(c = 0\) and positive \(g\) that \(f(x)/g(x) \to 0\).

\begin{Prop}
  Let \(i \in I\) and \(f(i) = j\).
  Suppose
  \[
  \Psi^{(i)}(q) \sim c q^{\alpha}K(q) \quad \text{and} \quad 
  \psi^{(j)}(q) \sim c' q^{\beta}H(q) \quad (q \to 0)
  \]
  for \(c \geq 0\), \(c' > 0\), \(\alpha,\beta \in [0,1]\), and slowly varying functions \(K,H\) at \(0\).
  Then we have
  \begin{align}    
  \Psi^{\ast,(i)}(q) \sim C q^{\alpha \beta}G(q) \quad (q \to 0), \label{Eq07}
  \end{align}
  where 
  \[
  C := c(c')^{\alpha} \quad \text{and} \quad G(q) := H(q)^{\alpha}K(q^{\beta}H(q)).
  \]
  In particular, \(G\) is slowly varying at \(0\), and whenever \(c > 0\), the function \(\Psi^{\ast,(i)}\) is regularly varying at \(0\) with exponent \(\alpha\beta\).
\end{Prop}

\begin{proof}
By Proposition \ref{prop:LEofTimeChange},
\[
\Psi^{\ast,(i)}(q)=\Psi^{(i)}(\psi^{(j)}(q)).
\]
Hence,
\[
\Psi^{\ast,(i)}(q)
\sim
c (\psi^{(j)}(q))^\alpha K(\psi^{(j)}(q))
\sim
c(c')^\alpha q^{\alpha\beta}G(q).
\]
The slow variation of \(G\) follows from \cite[Proposition 1.5.7(ii)]{Regularvariation}.
\end{proof}

To prepare for the scaling limits considered later, we recall the following classical criterion for convergence of subordinators to a stable subordinator.

\begin{Prop}\label{prop:stableConvergence}
Let \(Y\) be a subordinator with L\'evy measure \(\nu\) and Laplace exponent \(\psi\).
Let \(\alpha\in(0,1)\), let \(c\ge0\), and let \(K\) be a slowly varying function at \(\infty\).
Define
\[
h(x):=\frac{x^\alpha}{\Gamma(1-\alpha)K(x)}
\quad (x>0),
\]
and let \(g\) be an asymptotic inverse of \(h\) (see e.g., \cite[Theorem 1.5.12]{Regularvariation} for the definition of asymptotic inverses).
Let \(Z^{(\alpha)}\) be the \(\alpha\)-stable subordinator characterized by
\[
E[e^{-qZ^{(\alpha)}_t}]=e^{-tq^\alpha}
\quad (q,t\ge0).
\]
Then the following are equivalent:
\begin{enumerate}
  \item \(\left(\frac{Y_{\lambda t}}{g(\lambda)}\right)_{t\ge0}
  \xrightarrow[\lambda\to\infty]{d}
  \bigl(c^{1/\alpha}Z^{(\alpha)}_t\bigr)_{t\ge0}\) on \(D([0,\infty),\bR)\), equipped with the \(J_1\)-topology.
  \item
  \(\nu(x,\infty)\sim c x^{-\alpha}K(x) \quad (x\to\infty)\).
  \item
  \(\psi(q)\sim c \Gamma(1-\alpha) q^\alpha K(1/q) \quad (q \to 0)\).
\end{enumerate}
\end{Prop}

\begin{proof}
  We only give a sketch.
  Since \(Y\) is a L\'evy process, the convergence on \(D([0,\infty),\bR)\) is equivalent to the convergence of one-dimensional marginal distributions.
  Then the equivalence between (i) and (iii) follows by computing the Laplace transform.
  The equivalence between (ii) and (iii) follows from the L\'evy-Khintchine representation of the Laplace exponent and Karamata's Tauberian theorem and the monotone density theorem (see e.g., \cite[Theorems 1.7.1 and 1.7.2]{Regularvariation}).
\end{proof}

\subsection{Occupation times}

We next define the occupation times of \(X^\ast\) in terms of its excursions.
For \(i\in I\), the occupation time of the class \(D_0^{(i)}\) counts the amount of time spent in excursions belonging to \(D_0^{(i)}\).
Up to a fixed time \(t\), this consists of the total lifetimes of the excursions already completed before \(t\), together with the elapsed time of the excursion in progress at time \(t\), if any.

The excursions of \(X^{\ast}\) completed strictly before time \(t\) corresponds to atoms in \([0,L(A_t))\times D_0\).
Their contribution to the \(i\)-th occupation time is
\[
\int_{[0,L(A_t))\times D_0}\zeta(e) N^{\ast,(i)}(ds\,de).
\]
If \(X^\ast\) is away from \(o\) at time \(t\), then \(X_{A_t}\) is away from \(o\), and there is a corresponding atom \((L(A_t),e)\in \Supp N^{\ast}\).
By Proposition \ref{prop:excursionIntervalCorrespondence}, its excursion interval is obtained from the corresponding excursion interval of \(X\) by the time change, and hence its departure time is \(T_{\eta_{L(A_t)-}}\) almost surely; see also Remark \ref{rem:EndpointOfExcursionInterval}.
Thus, the elapsed time of the current excursion is \(t-T_{\eta_{L(A_t)-}}\).
We therefore define
\begin{align}
  \begin{aligned}
    O^{\ast,(i)}(t) :=&
    \int_{[0,L(A_t))\times D_0}\zeta(e) N^{\ast,(i)}(ds\,de) \\
    &\quad+
    \left(t-T_{\eta_{L(A_t)-}}\right)
    1\{N^{\ast,(i)}(\{L(A_t)\}\times D_0)=1\}
  \end{aligned}
  \quad (i\in I,\ t\ge0).
  \label{Eq13}
\end{align}
The occupation time at \(o\) is defined by
\[
O^{\ast,(0)}(t):=\int_{0}^{t} 1\{X^{\ast}_{s} = o\} ds
\quad (t\ge0).
\]

A useful simplification occurs in the timescale of \(T(\eta_t)\).
For \(i\in I\), the definition immediately gives
\begin{align}
O^{\ast,(i)}(T(\eta_t))=\Gamma^{(i)}(t)
\quad (t\ge0).
\label{Eq11}
\end{align}
Thus, along the clock \(T(\eta_t)\), the occupation time of the \(i\)-th class coincides with the cumulative excursion lifetime process of that class.
Summing over \(i\in I\), we also have
\[
\sum_{i\in I}O^{\ast,(i)}(T(\eta_t))=\Gamma(t).
\]

It remains to identify \(O^{\ast,(0)}(T(\eta_t))\).
The following proposition reveals that it behaves as a subordinator determined solely by the stagnancy rate and the subordinator assigned to \(o\), independent of the excursion lifetimes \((\Gamma^{(i)})_{i\in I}\).

\begin{Prop}\label{prop:Gamma0LawAndIndependence}
Define
\[
\Gamma^{(0)}_t:=O^{\ast,(0)}(T(\eta_t))
\quad (t\ge0).
\]
Then
\begin{align}  
\Gamma^{(0)}_t=T(\eta_t)-\Gamma(t)
\quad (t\ge0). \label{Eq14}
\end{align}
Moreover, for \(j_0:=f(0)\), the process \((\Gamma^{(0)}_t)_{t\ge0}\) has the same law on \(D([0,\infty),\bR)\) as \((Z^{(j_0)}(ct))_{t\ge0}\).
In particular, \((\Gamma^{(0)}_t)_{t\ge0}\) is a subordinator.

Furthermore, \(\Gamma^{(0)}\) is independent of the transformed
excursion point process \(N^\ast\), and hence of the family
\((\Gamma^{(i)})_{i\in I}\).
\end{Prop}

For the proof, we use the following general lemma on subordinators.

\begin{Lem}\label{lem:subordinatorRandomMeasure}
Let \(Z\) be a subordinator without killing, with Laplace exponent
\(\psi\), and let \(\mu_Z\) be the random measure on \([0,\infty)\)
defined by
\[
\mu_Z(\{0\})=0,
\quad
\mu_Z((a,b])=Z_b-Z_a
\quad(0\le a<b<\infty).
\]
If \(A\subset[0,\infty)\) is a Borel set with \(\lambda(A)<\infty\),
then
\begin{equation}\label{Eq08}
\bE[e^{-q\mu_Z(A)}]
=
e^{-\lambda(A)\psi(q)}
\quad(q\ge0).
\end{equation}
In particular,
\[
\mu_Z(A)\overset{d}{=}Z_{\lambda(A)},
\]
and the values of \(\mu_Z\) on disjoint Borel sets of finite
Lebesgue measure are independent.

More generally, let \(\cH\) be a \(\sigma\)-field independent
of \(Z\), and let \(A_1,\ldots,A_n\) be disjoint
\(\cH\)-measurable random Borel sets satisfying
\(\lambda(A_m)<\infty\) almost surely.
Then, for \(q_1,\ldots,q_n\ge0\),
\begin{equation}\label{eq:subordinator-random-measure-joint}
\bE\left[
\exp\left\{-\sum_{m=1}^nq_m\mu_Z(A_m)\right\}~\middle|~\cH
\right]
=
\exp\left(-\sum_{m=1}^n\lambda(A_m)\psi(q_m)\right)
\quad\text{a.s.}
\end{equation}
Thus, the random variables \(\mu_Z(A_1),\ldots,\mu_Z(A_n)\)
are conditionally independent given \(\cH\).
In particular, if \(\lambda(A_1)=a\) almost surely for a constant
\(a\ge0\), then \(\mu_Z(A_1)\overset{d}{=}Z_a\).
\end{Lem}

\begin{proof}
When \(A\) is a finite disjoint union of bounded left-open right-closed intervals,
we readily have \eqref{Eq08} since \(Z\) has stationary independent increments.
By a usual approximation argument, \eqref{Eq08} can be then extended to the disjoint countably infinite union of left-open intervals, to open sets, and to Borel sets.
The desired independence property also follows from the same argument.
For the conditional version of the assertions, the same arguments under conditioning are valid.
The remaining assertions follow immediately.
\end{proof}




\begin{proof}[Proof of Proposition \ref{prop:Gamma0LawAndIndependence}]
Since
\[
u=O^{\ast,(0)}(u)+\sum_{i\in I}O^{\ast,(i)}(u)
\qquad(u\ge0),
\]
substituting \(u=T(\eta_t)\) gives
\begin{align}  
\Gamma_t^{(0)}=T(\eta_t)-\Gamma(t). \label{Eq23}
\end{align}

We identify the law of \((\Gamma^{(0)}_t)_{t\ge0}\).
To this end, we represent \(T(\eta_t)\) and \(\Gamma(t)\) by the random measures \(\mu_{Z^{(j)}}\).
Set \(j_0:=f(0)\).
For each \(j\in J\), define
\[
\tilde{N}^{(j)}
:=
\sum_{\substack{i\in I\\f(i)=j}}N^{(i)},
\]
and recall that
\[
B_t^{(j)}
=
ct\,1_{\{j=j_0\}}
+
\int_{[0,t]\times D_0}\zeta(e)\,
\tilde{N}^{(j)}(du\,de).
\]
For each atom \((u,e)\in\Supp\tilde{N}^{(j)}\), put
\[
I_u^{(j)}:=(B_{u-}^{(j)},B_u^{(j)}],
\]
and define
\[
E_t^{(j)}
:=
\bigsqcup_{\substack{(u,e)\in\Supp\tilde{N}^{(j)}\\u\le t}}
I_u^{(j)},
\quad
C_t^{(j)}:=(0,B_t^{(j)}]\setminus E_t^{(j)}.
\]
Then, by construction,
\[
\sum_{\substack{i \in I \\f(i)=j}}\Gamma^{(i)}(t)=\mu_{Z^{(j)}}(E_t^{(j)}).
\]
Hence,
\begin{align}  
\Gamma(t)=\sum_{j\in J}\mu_{Z^{(j)}}(E_t^{(j)}), \quad T(\eta_t)=\sum_{j\in J}\mu_{Z^{(j)}}((0,B_t^{(j)}]). \label{Eq24}
\end{align}

Since the intervals \(I_u^{(j)}\) are disjoint and have lengths
\(\zeta(e)\), we have, for every \(j\in J\),
\begin{align}
\lambda(C_t^{(j)})
=
B_t^{(j)}
-
\int_{[0,t]\times D_0}\zeta(e)\,
\tilde{N}^{(j)}(du\,de) = ct\,1_{\{j=j_0\}}.
\label{eq:Gamma0-remainder-length}
\end{align}
Note that for each \(j\in J\), the sets \(C_t^{(j)}\) are increasing
in \(t\), and
\begin{equation}\label{eq:Gamma0-disjoint-regions}
C_s^{(j)}\cap E_t^{(j)}=\emptyset
\qquad(s,t\ge0).
\end{equation}
Indeed, every interval \(I_u^{(j)}\) with \(u>s\) is contained in
\((B_s^{(j)},\infty)\), so later excursion intervals do not remove
any point from \(C_s^{(j)}\).

Hence, from \eqref{Eq23} and \eqref{Eq24}, we have
\begin{equation}\label{eq:Gamma0-remainder-representation}
\Gamma_t^{(0)}
=
\sum_{j\in J}\mu_{Z^{(j)}}(C_t^{(j)}).
\end{equation}

We now identify the conditional finite-dimensional distributions
of \(\Gamma^{(0)}\) given \(N\).
Fix \(0=t_0<t_1<\cdots<t_n\), and put
\[
C_{t_{m-1},t_m}^{(j)}
:=
C_{t_m}^{(j)}\setminus C_{t_{m-1}}^{(j)}
\quad(j\in J,\ 1\le m\le n).
\]
For each fixed \(j\), these sets are pairwise disjoint.
They are measurable with respect to \(N\), and
\[
\lambda(C_{t_{m-1},t_m}^{(j)})
=
c(t_m-t_{m-1})1_{\{j=j_0\}}.
\]
Furthermore, \eqref{eq:Gamma0-remainder-representation} implies
\[
\Gamma_{t_m}^{(0)}-\Gamma_{t_{m-1}}^{(0)}
=
\sum_{j\in J}
\mu_{Z^{(j)}}(C_{t_{m-1},t_m}^{(j)}).
\]

Conditionally on \(N\), the random measures \((\mu_{Z^{(j)}})_{j \in J}\) are independent,
and Lemma \ref{lem:subordinatorRandomMeasure} therefore gives for \(q_{m} \geq 0\)
\begin{align}
&\bE\left[
\exp\left(
-\sum_{m=1}^nq_m
\bigl(\Gamma_{t_m}^{(0)}-\Gamma_{t_{m-1}}^{(0)}\bigr)
\right)
~\middle|~N
\right]
\nonumber\\
&\quad=
\exp\left(
-\sum_{j\in J}\sum_{m=1}^n
\lambda(C_{t_{m-1},t_m}^{(j)})\psi^{(j)}(q_m)
\right) \\
&\quad=
\exp\left(
-c\sum_{m=1}^n(t_m-t_{m-1})\psi^{(j_0)}(q_m)
\right)
\quad\text{a.s.}
\label{eq:Gamma0-conditional-increments}
\end{align}
When \(J\) is infinite, the same identity follows by taking finite
partial sums and applying conditional dominated convergence.

The right-hand side is deterministic and is the joint Laplace
transform of the corresponding increments of
\((Z^{(j_0)}_{ct})_{t\ge0}\).
Since \(\Gamma^{(0)}\) has c\`adl\`ag paths and starts at zero,
\eqref{eq:Gamma0-conditional-increments} shows that
\(\Gamma^{(0)}\) has the same law on \(D([0,\infty),\mathbb R)\)
as \((Z^{(j_0)}_{ct})_{t\ge0}\).
In particular, it is a subordinator.
The same identity also shows that its conditional law given \(N\)
does not depend on \(N\). Thus, for every bounded measurable
functional \(F\),
\begin{equation}\label{eq:Gamma0-independent-of-N}
\bE[F(\Gamma^{(0)})\mid N]
=
\bE[F(\Gamma^{(0)})]
\quad\text{a.s.}
\end{equation}

Finally, we prove the independence of \(\Gamma^{(0)}\) from \(N^{\ast}\).
Conditionally on \(N\), \eqref{eq:Gamma0-remainder-representation} shows that \(\Gamma^{(0)}\) is determined by the restrictions of the random measures \((\mu_{Z^{(j)}})_{j\in J}\) to the sets \((C_t^{(j)})_{j\in J,\ t \geq 0}\).
On the other hand, conditionally on \(N\), each atom of \(N^{\ast}\) is determined by the restrictions of \((\mu_{Z^{(j)}})_{j\in J}\) to the sets \((E^{(j)}_{t})_{j \in J,t \geq 0}\) by Lemma \ref{Lem:excursionwise-representation}.

By \eqref{eq:Gamma0-disjoint-regions}, these two collections of
regions are disjoint for each fixed \(j\).
The independence of restrictions of each random measure \(\mu_{Z^{(j)}}\),
together with the independence of \((Z^{(j)})_{j \in J}\), therefore
implies that \(\Gamma^{(0)}\) and \(N^{\ast}\) are conditionally independent given \(N\).

Let \(F\) be a bounded measurable functional of \(\Gamma^{(0)}\),
and let \(G\) be a bounded measurable functional of \(N^\ast\).
Conditional independence and \eqref{eq:Gamma0-independent-of-N} yield
\begin{align*}
\bE\left[
F(\Gamma^{(0)})G(N^{\ast})
\right]
&=
\bE\left[
\bE[F(\Gamma^{(0)})\mid N]
\bE[G(N^{\ast})\mid N]
\right]\\
&=
\bE[F(\Gamma^{(0)})]
\bE[G(N^{\ast})].
\end{align*}
This proves the desired independence.
\end{proof}

\begin{Rem}
The independence of \(\Gamma^{(0)}\) and \(\Gamma\) is already suggested by Proposition \ref{prop:TetaIsSubordinator}.
However, the additive decomposition of the Laplace exponents not necessarily imply that the given processes \(\Gamma^{(0)}\) and \(\Gamma\) are independent.
\end{Rem}

\begin{Cor} \label{cor:zeroStagnancyPreservation}
If the stagnancy rate \(c\) at \(o\) of \(X\) is zero, then the process
\[
(O^{\ast,(0)}(t))_{t\ge0}
\]
is identically zero almost surely.
\end{Cor}

The following formula, known as Williams formula, is useful since it reduces the occupation time to cumulative excursion lifetimes, which plays a fundamental role in our treatment of occupation time limit theorem in Section \ref{section:occupationTimeLimitTheorem}.

\begin{Prop}[Williams formula]
Let \(i\in I_{0}\).
For \(0\leq t < O^{\ast,(i)}(\infty)\), we have 
\begin{align}  
(O^{\ast,(i)})^{-1}(t) = t + \sum_{k\in I_{0},k\neq i}\Gamma^{(k)}(\Lambda^{(i)}_t), \label{Eq09}
\end{align}
where we denote the right-continuous inverse of \(\Gamma^{(i)}\) by \(\Lambda^{(i)}\) for \(i \in I_{0}\).
\end{Prop}

\begin{proof}
The processes \((\Gamma^{(k)})_{k\in I_0}\) are independent
subordinators and therefore have no common jump times almost surely.
We work on this probability-one event.

Write
\[
\cT_s:=T(\eta_s)=\sum_{k\in I_0}\Gamma^{(k)}(s).
\]

Fix \(i\in I_0\) and \(0\le t<O^{\ast,(i)}(\infty)\), and set
\[
u:=(O^{\ast,(i)})^{-1}(t),
\quad
r:=\Lambda^{(i)}_t.
\]
Then, clearly \(u,r<\infty\).
Since \(O^{\ast,(i)}\) is continuous, it follows that \(O^{\ast,(i)}(u)=t\).

For \(s<r\), we have
\[
O^{\ast,(i)}(\cT_s)=\Gamma^{(i)}(s)\le t,
\]
and hence \(\cT_s\le u\).
For \(s>r\), we have
\[
O^{\ast,(i)}(\cT_s)=\Gamma^{(i)}(s)>t,
\]
and hence \(u\le\cT_s\).
Taking limits gives
\begin{align}  
\cT_{r-}\le u\le\cT_r, \label{Eq25}
\end{align}
with the convention \(\cT_{0-}=0\).

If \(\Gamma^{(i)}(r)=t\), then \(O^{\ast,(i)}(\cT_r)=t\).
Then the definition of \(u\) gives \(\cT_{r}\leq u\).
Thus, we have \(u=\cT_r\) from \eqref{Eq25}.
That also shows
\[
O^{\ast,(k)}(u)=\Gamma^{(k)}(r)
\quad(k\ne i)
\]
from \eqref{Eq11}.

Otherwise, \(\Gamma^{(i)}(r)>t\).
Since \(\Gamma^{(i)}(r-)\le t\), the process \(\Gamma^{(i)}\) jumps at \(r\).
All the other components \((\Gamma^{(k)})_{k \neq i}\) are therefore continuous at \(r\).
By \eqref{Eq11} together with the corresponding identity for \(k=0\),
we  have
\[
O^{\ast,(k)}(\cT_{r-})
=
\Gamma^{(k)}(r-)
=
\Gamma^{(k)}(r)
=
O^{\ast,(k)}(\cT_r)
\quad(k\ne i).
\]
Then it follows from \eqref{Eq25} that
\[
O^{\ast,(k)}(u)=\Gamma^{(k)}(r)
\quad(k\ne i).
\]

In either case, we obtain
\[
u
=
\sum_{k\in I_0}O^{\ast,(k)}(u)
=
t+\sum_{\substack{k\in I_0\\k\ne i}}\Gamma^{(k)}(r).
\]
The proof is complete.
\end{proof}

\section{Multiscale limit theorems for occupation times}
\label{section:occupationTimeLimitTheorem}

In this section, we derive limit theorems for the occupation times of
\(X^\ast\).
We continue to assume that the measurable partition
\((D_0^{(i)})_{i\in I}\) is preserved by the time change.

The main result of this section is Theorem \ref{thm:occupation_limit}.
It describes the joint scaling limit of the occupation times \((O^{\ast,(i)})_{i\in I_0}\) under regular variation assumptions.
The result gives a multiscale asymptotics of occupation times.

Our strategy is first to establish scaling limits for the cumulative excursion lifetime processes \((\Gamma^{(i)})_{i\in I_0}\), and then to express the occupation times in terms of these processes.
The main technical point is that the relevant operations involve inverses and compositions of c\`adl\`ag non-decreasing functions.
Therefore, before proving the main theorem, we record several elementary facts on such functions and Skorokhod's \(J_1\)-topology.

\subsection{Auxiliary lemmas for \(J_1\)-convergence}

For \(d\in\bN\), we denote by \(D([0,\infty),\bR^d)\) the space of
c\`adl\`ag functions from \([0,\infty)\) to \(\bR^d\), equipped with
Skorokhod's \(J_1\)-topology.
We equip \(C([0,\infty),\bR^d)\), the space of continuous functions
from \([0,\infty)\) to \(\bR^d\), with the topology of local uniform
convergence.

We do not recall the full definition of the \(J_1\)-topology;
see, for example, \cite[Chapter VI]{JacodShiryaev}.
We only record the form of the convergence criterion that will be
used below.

Let \(\Lambda\) be the set of continuous strictly increasing maps
\(\lambda:[0,\infty)\to[0,\infty)\) such that
\[
\lambda(0)=0,
\quad
\lambda(t)\to\infty \quad (t\to\infty).
\]
Then \(x_n\to x\) on \(D([0,\infty),\bR^d)\) with the \(J_1\)-topology if and
only if, for every \(T>0\), there exist \(\lambda_n\in\Lambda\) such that
\[
\sup_{t\in[0,T]}
\left|x_n(\lambda_n(t))-x(t)\right|
\longrightarrow0,
\quad
\sup_{t\in[0,T]}
|\lambda_n(t)-t|
\longrightarrow0.
\]

We note the following elementary but important distinction.
Although the product space \(\prod_{i=1}^d D([0,\infty),\bR)\) can be naturally identified with \(D([0,\infty),\bR^d)\) as a set, the product topology is not the same as the \(J_1\)-topology on \(D([0,\infty),\bR^d)\).
The latter is stronger, because the same time-change \(\lambda_{n}\) must work simultaneously for all coordinates.
In fact, the \(J_{1}\)-topology on \(D([0,\infty),\bR^{d})\) is strictly stronger, as illustrated by the following simple example.
We always equip the product space with the product topology.

\begin{Ex}
Let
\[
y(t):=1_{[1,\infty)}(t),
\quad
y_n(t):=y(t),
\quad
z_n(t):=1_{[1+1/n,\infty)}(t),
\]
and set
\[
x_n:=(y_n,z_n),
\quad
x:=(y,y).
\]
Then \(y_n\to y\) and \(z_n\to y\) on \(D([0,\infty),\bR)\), and hence
\(x_n\to x\) on \(D([0,\infty),\bR)\times D([0,\infty),\bR)\).

On the other hand, \(x_n\) does not converge to \(x\) on
\(D([0,\infty),\bR^{2})\).
Indeed, any single time change \(\lambda \in \Lambda\) must send the jump time of \(y_{n}\) and the jump time of \(z_{n}\) to two distinct times.  More explicitly, let
\[
a:=\lambda^{-1}(1),
\quad
b:=\lambda^{-1}(1+1/n).
\]
Then \(a<b\).
For every \(s\in[a,b)\), we have
\[
y_{n}(\lambda(s))=1,
\quad
z_{n}(\lambda(s))=0,
\]
and therefore
\[
|x_n(\lambda(s))-x(s)|\geq1.
\]
Thus, it follows that \(x_n\not\to x\) on \(D([0,\infty),\bR^{2})\) with the \(J_1\)-topology.
\end{Ex}

Below, we collect several continuity properties of the mappings appearing in our analysis, particularly those involving right-continuous inverses and compositions.
Although these results are largely standard, we record their precise statements here for convenience.
The proofs of Lemmas \ref{lem:dense_pointwise_inverse}, \ref{lem:inverse_composition_pointwise}, and \ref{lem:coordinatewise_to_multivariate_J1} are deferred to Appendix \ref{appendix:proofs}.

We begin with a deterministic lemma on the convergence of right-continuous inverses of non-decreasing functions.
The point is that the convergence of inverses only requires pointwise convergence on a dense set, rather than \(J_1\)-convergence.

\begin{Lem}\label{lem:dense_pointwise_inverse}
Let \(f_n,f:[0,\infty)\to[0,\infty)\) be non-decreasing functions and suppose
\[
f(0)=0,
\quad
f(\infty)=\infty.
\]
Assume that there exists a dense subset \(D\subset[0,\infty)\) such that
\[
f_n(x)\to f(x)
\quad (n\to\infty)
\]
for every \(x\in D\).
Then, for every continuity point \(t>0\) of \(f^{-1}\), and also for \(t=0\) if \(f^{-1}(0)=0\) we have
\[
f_n^{-1}(t)\to f^{-1}(t)
\quad (n\to\infty),
\]
where \(f_n^{-1}\) and \(f^{-1}\) denote the right-continuous inverses.
Moreover, if \(f\) is strictly increasing,
\[
\sup_{u\in[0,T]}|f_n^{-1}(u)-f^{-1}(u)|\to0
\quad (n\to\infty)
\]
for every \(T>0\).
\end{Lem}

We next combine the previous lemma with \(J_1\)-convergence.

\begin{Lem} \label{lem:inverse_composition_pointwise}
Let \(x_n,y_n,x,y\in D([0,\infty),\bR)\).
Assume the following:
\begin{enumerate}
  \item \(x_n\to x\) on \(D([0,\infty),\bR)\) as \(n\to\infty\),
  \item \(y_n\) and \(y\) are non-decreasing, with
  \[
  y(0)=0,
  \quad
  y_n(\infty)=y(\infty)=\infty,
  \]
  \item \(y\) is strictly increasing,
  \item there exists a dense subset \(D\subset[0,\infty)\) such that
  \[
  y_n(u)\to y(u)
  \quad (n\to\infty)
  \]
  for every \(u\in D\).
\end{enumerate}
Then, for every \(t\ge0\) such that \(x\) is continuous at \(y^{-1}(t)\), we have
\begin{align}
x_n(y_n^{-1}(t))\longrightarrow x(y^{-1}(t))
\quad (n\to\infty),
\label{Eq10}
\end{align}
where \(y_n^{-1}\) and \(y^{-1}\) denote the right-continuous inverses.
\end{Lem}




We also need a criterion that upgrades the convergence on the product space \(\prod_{1 \leq i \leq d}D([0,\infty),\bR)\) to that on \(D([0,\infty),\bR^{d})\).

\begin{Lem}\label{lem:coordinatewise_to_multivariate_J1}
Let \(d \in \bN\) and let \(x_n^{(i)},x^{(i)}\in D([0,\infty),\bR) \ (1 \leq i \leq d, \ n \in \bN)\).
Assume that, for every \(1 \leq i \leq d\),
\[
x_n^{(i)}\to x^{(i)}
\quad \text{on } D([0,\infty),\bR).
\]
Assume moreover that the limits \((x^{(i)})_{1 \leq i \leq d}\) have no common jump times, that is,
\[
\Delta x^{(i)}(t)\Delta x^{(j)}(t)=0
\quad
\text{for every } t\ge0,\ i\neq j.
\]
Then
\[
(x_n^{(i)})_{1 \leq i \leq d}\to (x^{(i)})_{1 \leq i \leq d}
\quad
\text{on } D([0,\infty),\bR^{d}).
\]
\end{Lem}

The corresponding probabilistic statement will be used repeatedly.

\begin{Cor}\label{cor:product_to_multivariate_J1}
Let \(d \in \bN\) and let \(X_n^{(i)},X^{(i)} \ (1 \leq i \leq d,\ n \in \bN)\) be \(D([0,\infty),\bR)\)-valued random elements.
Assume that
\[
(X_n^{(i)})_{1 \leq i \leq d}\xrightarrow[n\to\infty]{d}(X^{(i)})_{1 \leq i \leq d}
\]
on the product space \(\prod_{1 \leq i \leq d}D([0,\infty),\bR)\).
Assume moreover that the limits \((X^{(i)})_{1 \leq i \leq d}\) have no common jump times almost surely.
Then
\[
(X_n^{(i)})_{1 \leq i \leq d}\xrightarrow[n\to\infty]{d}(X^{(i)})_{1 \leq i \leq d}
\]
on \(D([0,\infty),\bR^{d})\).
\end{Cor}

\begin{proof}
This follows from Skorokhod's representation theorem and Lemma \ref{lem:coordinatewise_to_multivariate_J1}.
\end{proof}

\subsection{Multiscale limits for occupation times}

Define
\begin{align}
\rho^{(i)}(dt)=
\begin{cases}
n^{\ast,(i)}(\zeta\in dt) &(i\in I),\\
c\,\nu_{f(0)}(dt) &(i=0).
\end{cases}
\label{def:measureRho}
\end{align}
They are the L\'evy measures of the subordinators \(\Gamma^{(i)}\).

Throughout this section, we assume that \(I_{0}\) is finite and that the following \textit{dominant regular variation condition} holds: there exist \(\alpha\in(0,1)\), a positive slowly varying function \(K\) at \(\infty\), and constants \(\beta_i\ge0\) \((i\in I_{0})\) with \(\sum_{i\in I_{0}}\beta_i=1\) such that
\begin{equation}
\tag{\(A\)}\label{assumption:regularVariation}
\rho^{(i)}((t,\infty))\sim \beta_i t^{-\alpha}K(t)
\quad (t\to\infty)
\end{equation}
for every \(i\in I_{0}\).

For \(i\in I_{0}\) with \(\beta_i=0\), we consider the following \textit{sub-dominant regular variation condition}: there exist \(\tilde{\beta}_i>0\), \(\tilde{\alpha}_i\in[\alpha,1)\), and a slowly varying function \(K_i\) at \(\infty\) such that
\begin{equation}
\tag{\(B_i\)}\label{assumption:regularVariation2}
\rho^{(i)}((t,\infty))\sim \tilde{\beta}_i t^{-\tilde{\alpha}_i}K_i(t)
\quad (t\to\infty).
\end{equation}

Accordingly, we decompose \(I_0\) as
\[
I_0=I_{0,\mathrm{dom}}\sqcup I_{0,\mathrm{sub}}\sqcup I_{0,\mathrm{res}},
\]
where
\begin{align*}
I_{0,\mathrm{dom}}
&:=\{i\in I_0 \mid \beta_i>0\},\\
I_{0,\mathrm{sub}}
&:=\{i\in I_0 \mid \beta_i=0 \text{ and } i \text{ satisfies \eqref{assumption:regularVariation2}}\},\\
I_{0,\mathrm{res}}
&:=I_0\setminus\bigl(I_{0,\mathrm{dom}}\sqcup I_{0,\mathrm{sub}}\bigr).
\end{align*}
We allow that \(I_{0,\mathrm{sub}}\) or \(I_{0,\mathrm{res}}\) may be empty, whereas \(I_{0,\mathrm{dom}}\) is not by \(\sum_{i\in I_0}\beta_i=1\).

\begin{Rem}
If \(i\in I_0\) satisfies \(\beta_i=0\) and \eqref{assumption:regularVariation2}, then
\[
t^{\tilde{\alpha}_i-\alpha}\frac{K(t)}{K_i(t)}\to\infty
\quad (t\to\infty).
\]
Thus, every sub-dominant class is negligible on the dominant scale \(t^{-\alpha}K(t)\).
\end{Rem}

We begin with the scaling limit of the cumulative excursion lifetime processes.

\begin{Prop}\label{prop:GammaScaling}
Let \(g\) be an asymptotic inverse of the function
\[
h(x) := \frac{x^\alpha}{\Gamma(1-\alpha)K(x)},
\]
and, for each \(i\in I_{0,\mathrm{sub}}\), let \(g_{i}\) be an asymptotic inverse of the function
\[
x\longmapsto \frac{x^{\tilde{\alpha}_i}}{\Gamma(1-\tilde{\alpha}_i)K_i(x)}.
\]
For \(i\in I_0\), define
\begin{align}
X^{(i)}_\lambda(t):=
\begin{cases}
\frac{\Gamma^{(i)}_{\lambda t}}{g(\lambda)} & (i\in I_{0,\mathrm{dom}}\cup I_{0,\mathrm{res}}),\\
\frac{\Gamma^{(i)}_{\lambda t}}{g_i(\lambda)} & (i\in I_{0,\mathrm{sub}}),
\end{cases}
\quad (t\ge0),
\end{align}
and
\begin{align}  
X^{(i)}(t):=
\begin{cases}
\xi^{(i)}(t) & (i\in I_{0,\mathrm{dom}}),\\
\tilde{\xi}^{(i)}(t) & (i\in I_{0,\mathrm{sub}}),\\
0 & (i\in I_{0,\mathrm{res}}),
\end{cases}
\quad (t\ge0),
\end{align}
where the family
\[
\bigl((\xi^{(i)})_{i\in I_{0,\mathrm{dom}}},(\tilde{\xi}^{(i)})_{i\in I_{0,\mathrm{sub}}}\bigr)
\]
consists of independent stable subordinators characterized by
\[
E[e^{-q\xi^{(i)}(t)}]=\exp(-t\beta_i q^\alpha)
\quad (i\in I_{0,\mathrm{dom}}),
\]
and
\[
E[e^{-q\tilde{\xi}^{(i)}(t)}]=\exp(-t\tilde{\beta}_i q^{\tilde{\alpha}_i})
\quad (i\in I_{0,\mathrm{sub}}),
\]
for \(q,t\ge0\).
Then
\[
\bigl(X^{(i)}_\lambda\bigr)_{i\in I_0}
\xrightarrow[\lambda\to\infty]{d}
\bigl(X^{(i)}\bigr)_{i\in I_0}
\]
on \(D([0,\infty),\bR^{I_0})\).
\end{Prop}

\begin{proof}
For each \(i\in I_{0,\mathrm{dom}}\), Proposition \ref{prop:stableConvergence} yields
\[
X^{(i)}_\lambda
=
\left(\frac{\Gamma^{(i)}_{\lambda t}}{g(\lambda)}\right)_{t\ge0}
\xrightarrow[\lambda\to\infty]{d}
(\xi^{(i)}(t))_{t\ge0}
=
X^{(i)}
\]
on \(D([0,\infty),\bR)\).
Similarly, for each \(i\in I_{0,\mathrm{sub}}\), Proposition \ref{prop:stableConvergence} gives
\[
X^{(i)}_\lambda
=
\left(\frac{\Gamma^{(i)}_{\lambda t}}{g_i(\lambda)}\right)_{t\ge0}
\xrightarrow[\lambda\to\infty]{d}
(\tilde{\xi}^{(i)}(t))_{t\ge0}
=
X^{(i)}
\]
on \(D([0,\infty),\bR)\).
Finally, for each \(i\in I_{0,\mathrm{res}}\), Proposition \ref{prop:stableConvergence} with \(c=0\) implies
\[
X^{(i)}_\lambda
=
\left(\frac{\Gamma^{(i)}_{\lambda t}}{g(\lambda)}\right)_{t\ge0}
\xrightarrow[\lambda\to\infty]{d}
0
=
X^{(i)}
\]
on \(D([0,\infty),\bR)\).

Since the processes \((\Gamma^{(i)})_{i\in I_0}\) are independent, it follows that
\[
\bigl(X^{(i)}_\lambda\bigr)_{i\in I_0}
\xrightarrow[\lambda\to\infty]{d}
\bigl(X^{(i)}\bigr)_{i\in I_0}
\]
on the product space \(\prod_{i\in I_0}D([0,\infty),\bR)\).

All non-constant limit processes are independent subordinators and hence have no fixed jump times.
Therefore, they have no common jump times almost surely.
Applying Corollary \ref{cor:product_to_multivariate_J1}, we obtain
\[
\bigl(X^{(i)}_\lambda\bigr)_{i\in I_0}
\xrightarrow[\lambda\to\infty]{d}
\bigl(X^{(i)}\bigr)_{i\in I_0}
\]
on \(D([0,\infty),\bR^{I_0})\).
\end{proof}


We now prove our main result, a multiscale occupation time limit theorem.
A characterization of the limit distributions is given in Remarks \ref{rem:skewBessel}, \ref{rem:limit_distribution}.

\begin{Thm}\label{thm:occupation_limit}
Define
\[
S(t) := \sum_{k\in I_{0,\mathrm{dom}}}\xi^{(k)}(t) \quad (t \ge 0).
\]
For \(i\in I_{0,\mathrm{dom}}\), let \(\ell^{(i)}\) denote the right-continuous inverse of \(\xi^{(i)}\), and define \(\Phi^{(i)}\) as the right-continuous inverse of the strictly increasing process
\[
t \longmapsto t+\sum_{k\in I_{0,\mathrm{dom}},\ k\neq i}\xi^{(k)}(\ell^{(i)}_t)
\quad (t\ge0).
\]
For \(i\in I_{0,\mathrm{sub}}\), define
\[
\Phi^{(i)}(t)
:=
\tilde{\xi}^{(i)}(S^{-1}_{t})
\quad (t\ge0).
\]
For \(i\in I_{0,\mathrm{res}}\), set
\[
\Phi^{(i)} \equiv 0.
\]

Then, for any sequence of fixed times \(0 \leq t_{1} < \cdots < t_{n}\), 
\[
\left(
\left(
\frac{O^{\ast,(i)}(\lambda \cdot)}{\lambda}
\right)_{i\in I_{0,\mathrm{dom}}\cup I_{0,\mathrm{res}}}, \frac{L_{A(\lambda\cdot)}}{h(\lambda)},
\left(
\left(
\frac{O^{\ast,(i)}(\lambda t_{k})}{g_{i}(h(\lambda))}
\right)_{1 \leq k \leq n}
\right)_{i\in I_{0,\mathrm{sub}}}
\right)
\]
converges in distribution to
\[
\left(
\left(
\Phi^{(i)}(\cdot)
\right)_{i\in I_{0,\mathrm{dom}}\cup I_{0,\mathrm{res}}},S^{-1}(\cdot),
\left(
\left(
\Phi^{(i)}(t_k)
\right)_{1 \leq k \leq n}
\right)_{i\in I_{0,\mathrm{sub}}}
\right)
\]
as \(\lambda \to \infty\) on the product space
\[
C([0,\infty),\bR^{I_{0,\mathrm{dom}}\cup I_{0,\mathrm{res}}} \times \bR)
\times
(\bR^n)^{I_{0,\mathrm{sub}}}.
\]
\end{Thm}

\begin{Rem}
We make two observations regarding the statement of Theorem \ref{thm:occupation_limit}.
\begin{enumerate}
  \item For \(i \in I_{0,\mathrm{sub}}\), the finite-dimensional convergence cannot be strengthened to convergence in \(J_1\)-topology.
  The occupation time \(O^{\ast,(i)}(t)\) are continuous, while the limit \(\tilde{\xi}^{(i)}(S^{-1}_t)\) has jumps.
  Since the space of continuous functions is closed in the \(J_1\)-topology, convergence to a discontinuous limit is impossible.
  \item When the stagnancy rate \(c\) is zero, the occupation time at \(o\) is identically zero, and the corresponding assertion for \(O^{\ast,(0)}\) is trivial.
  In this case, the process \(L_{A_t}\), which corresponds to a local time, serves as a substitute for of \(O^{\ast,(0)}\), and the scaling limit of \(L_{A_t}\) gives analog of the occupation time limit at \(o\).
\end{enumerate}
\end{Rem}

\begin{proof}[Proof of Theorem\ref{thm:occupation_limit}]
We use the processes \(X_{\lambda}^{(i)}\), \(X^{(i)}\), \(\xi^{(i)}\), and \(\tilde{\xi}^{(i)}\) introduced in Proposition \ref{prop:GammaScaling}.
Fix a sequence of times \(0\le t_1<\cdots<t_m\).
We take an arbitrary sequence \(\lambda_n\to\infty\).
It suffices to show the desired convergence along this sequence.

We first observe an elementary asymptotic relation on normalizing sequences.
For \(i\in I_{0,\mathrm{sub}}\), the assumptions
\eqref{assumption:regularVariation} and
\eqref{assumption:regularVariation2} imply
\[
\frac{x^{\tilde\alpha_i}/\bigl(\Gamma(1-\tilde\alpha_i)K_i(x)\bigr)}
     {x^\alpha/\bigl(\Gamma(1-\alpha)K(x)\bigr)}
\longrightarrow\infty
\quad(x\to\infty).
\]
By comparison of asymptotic inverses, \(g_i(y)/g(y)\to0\).
Together with \(g(h(\lambda))\sim\lambda\), this gives
\[
\frac{g_i(h(\lambda))}{\lambda}\longrightarrow0.
\]
For \(i\in I_0\), put
\[
a_n^{(i)}:=
\begin{cases}
g(h(\lambda_n))
&(i\in I_{0,\mathrm{dom}}\cup I_{0,\mathrm{res}}),\\
g_i(h(\lambda_n))
&(i\in I_{0,\mathrm{sub}}),
\end{cases}
\quad
\theta_n^{(i)}:=\frac{a_n^{(i)}}{\lambda_n}.
\]
Discarding finitely many terms if necessary, we assume that these
normalizations are positive. Then
\begin{equation}\label{eq:occupation-scale-coefficients}
\theta_n^{(i)}\longrightarrow
\begin{cases}
1 &(i\in I_{0,\mathrm{dom}}\cup I_{0,\mathrm{res}}),\\
0 &(i\in I_{0,\mathrm{sub}}).
\end{cases}
\end{equation}

We fix a Skorokhod coupling for convergence in Proposition \ref{prop:GammaScaling} and conduct the subsequent discussion on the coupling.
Apply Proposition \ref{prop:GammaScaling} along the sequence \(h(\lambda_n)\to\infty\).
By Skorokhod's representation theorem, there are a probability space and the processes \(\widehat X_n=(\widehat X_n^{(i)})_{i\in I_0}\) on it such that, for each \(n\),
\begin{equation}\label{eq:occupation-coupled-lifetimes}
\bigl(\widehat X_n^{(i)}\bigr)_{i\in I_0}
\overset{d}{=}
\bigl(X_{h(\lambda_n)}^{(i)}\bigr)_{i\in I_0}
=
\left(
\frac{\Gamma^{(i)}_{h(\lambda_n)\cdot}}{a_n^{(i)}}
\right)_{i\in I_0},
\end{equation}
and
\[
\bigl(\widehat X_n^{(i)}\bigr)_{i\in I_0}
\longrightarrow
\bigl(X^{(i)}\bigr)_{i\in I_0}
\quad\text{a.s. in }D([0,\infty),\mathbb R^{I_0}),
\]
where we retain the notation \(X^{(i)},\xi^{(i)},\tilde\xi^{(i)}\) for the
copies of the limit processes on the probability space.

We construct all auxiliary processes from \(\widehat X_n^{(i)}\).
Define
\[
S_n(s):=\sum_{i\in I_0}\theta_n^{(i)}\widehat X_n^{(i)}(s),
\quad
R_n:=S_n^{-1}.
\]
Unless \(\Gamma^{(i)}\) is zero almost surely, its path is unbounded almost surely.
For each such \(i \in I_{0}\), define
\[
\ell_n^{(i)}:=(\widehat X_n^{(i)})^{-1},
\]
and
\begin{equation}\label{eq:occupation-coupled-Williams}
F_n^{(i)}(u)
:=
\theta_n^{(i)}u+
\sum_{\substack{k\in I_0\\k\ne i}}
\theta_n^{(k)}
\widehat X_n^{(k)}\bigl(\ell_n^{(i)}(u)\bigr),
\quad
\widehat O_n^{(i)}:=(F_n^{(i)})^{-1}.
\end{equation}
If \(\Gamma^{(i)}\) is zero almost surely, set \(\widehat O_n^{(i)}\equiv0\).
Such \(i \in I_{0}\) can only belong to \(I_{0,\mathrm{res}}\).

For each \(n\), these constructions are measurable functions of \((\widehat{X}^{(i)}_{n})_{i \in I_{0}}\).
We clarify their relation to the original processes.
Consider \(i \in I_{0}\) such that \(\Gamma^{(i)}\) is non-zero.
On the original probability space, the right-continuous inverse of
\(\Gamma^{(i)}_{h(\lambda_n)\cdot}/a_n^{(i)}\) is
\[
u\longmapsto
\frac{\Lambda^{(i)}_{a_n^{(i)}u}}{h(\lambda_n)},
\]
and Williams' formula \eqref{Eq09} gives
\[
\frac{(O^{\ast,(i)})^{-1}\left(a_n^{(i)}u\right)}{\lambda_n}
=
\theta_n^{(i)}u+
\sum_{\substack{k\in I_0\\k\ne i}}
\frac{\Gamma^{(k)}\left(\Lambda^{(i)}_{a_n^{(i)}u}\right)}{\lambda_n}.
\]
Taking the right-continuous inverse of this gives \(O^{\ast,(i)}(\lambda_n\cdot)/a_n^{(i)}\).
Together with
\[
T(\eta_s)=\sum_{i\in I_0}\Gamma^{(i)}(s),
\quad
(T\circ\eta)^{-1}=L\circ A,
\]
this proves the following equality of joint path laws for every
fixed \(n\):
\begin{align}
&\left(
\bigl(\widehat X_n^{(i)}\bigr)_{i\in I_0},
S_n,R_n,
\bigl(\widehat O_n^{(i)}\bigr)_{i\in I_0}
\right)
\nonumber\\
&\quad\overset{d}{=}
\left(
\left(\frac{\Gamma^{(i)}_{h(\lambda_n)\cdot}}{a_n^{(i)}}\right)_{i\in I_0},
\frac{T(\eta_{h(\lambda_n)\cdot})}{\lambda_n},
\frac{L_{A_{\lambda_n\cdot}}}{h(\lambda_n)},
\left(\frac{O^{\ast,(i)}(\lambda_n\cdot)}{a_n^{(i)}}\right)_{i\in I_0}
\right).
\label{eq:coupled-occupation-dictionary}
\end{align}

In particular, the processes \(\widehat O_n^{(i)}\) are continuous
and non-decreasing, and \(S_n\) is strictly increasing, unbounded,
and starts at zero, almost surely.
The occupation-time identities from
Section~\ref{section:cumulativeOccupationTimeBasics} also give
\begin{equation}\label{eq:occupation-coupled-clock-identity}
\widehat O_n^{(i)}(S_n(s))=\widehat X_n^{(i)}(s)
\qquad(s\ge0,\ i\in I_0).
\end{equation}
Thus, the process \(R_n\) is continuous, and
\begin{equation}\label{eq:occupation-coupled-upper-bound}
0\le\widehat O_n^{(i)}(t)
\le\widehat X_n^{(i)}(R_n(t))
\qquad(t\ge0,\ i\in I_0),
\end{equation}
since \(t\le S_n(R_n(t))\).

By \eqref{eq:occupation-scale-coefficients} and the joint
\(J_1\)-convergence of \(\widehat X_n\),
\[
\bigl(\theta_n^{(i)}\widehat X_n^{(i)}\bigr)_{i\in I_0}
\longrightarrow
\bigl(X^{(i)}1_{\{i\in I_{0,\mathrm{dom}}\}}\bigr)_{i\in I_0}
\quad\text{a.s. in }D([0,\infty),\mathbb R^{I_0}).
\]
Since addition of functions is continuous for the multivariate
\(J_1\)-topology, we obtain
\begin{equation}\label{Eq15}
S_n\longrightarrow
S=\sum_{k\in I_{0,\mathrm{dom}}}\xi^{(k)}
\quad\text{a.s. in }D([0,\infty),\mathbb R).
\end{equation}

Since \(\sum_{k\in I_0}\beta_k=1\), the process \(S\) is an \(\alpha\)-stable subordinator with Laplace exponent \(q^\alpha\).
In particular, its paths are strictly increasing and unbounded.
Lemma~\ref{lem:dense_pointwise_inverse} gives
\begin{equation}\label{eq:occupation-coupled-clock-inverse}
R_n\longrightarrow S^{-1}
\quad\text{a.s. on}\quad C([0,\infty),\bR).
\end{equation}

Hereafter, we work on a probability-one event on which all the preceding path identities and convergences hold, and on which all non-zero limit processes \(X^{(i)} \  (i\in I_{0})\), have strictly increasing, unbounded paths and no common jump times.
The arguments below are pathwise, except for the final verification
at the fixed observation times in Case 2.
Write \(\disc(f)\) for the set of discontinuity points
of a c\`adl\`ag function \(f\) in \((0,\infty)\).

\paragraph{Case 1: \(i\in I_{0,\mathrm{dom}}\).}
Fix \(i\in I_{0,\mathrm{dom}}\).
Since \(\widehat X_n^{(i)}\to\xi^{(i)}\) in \(J_1\) and
\(\xi^{(i)}\) is strictly increasing,
Lemma \ref{lem:dense_pointwise_inverse} implies
\[
\ell_n^{(i)}\longrightarrow\ell^{(i)}
\quad \text{local uniformly}.
\]
We next pass to the limit in the inverse occupation-time formula
\eqref{eq:occupation-coupled-Williams}.

If \(k\in I_{0,\mathrm{sub}}\cup I_{0,\mathrm{res}}\), then
\(\theta_n^{(k)}\widehat X_n^{(k)}\to0\) in \(J_1\).
Applying Lemma~\ref{lem:inverse_composition_pointwise} with
\[
x_n=\theta_n^{(k)}\widehat X_n^{(k)},
\qquad y_n=\widehat X_n^{(i)},
\qquad x=0,
\qquad y=\xi^{(i)},
\]
we obtain
\[
\theta_n^{(k)}\widehat X_n^{(k)}(\ell_n^{(i)}(u))\longrightarrow0
\qquad(u\ge0).
\]
If \(k\in I_{0,\mathrm{dom}}\setminus\{i\}\), the same lemma,
now with \(x=\xi^{(k)}\), gives
\begin{equation}\label{Eq12}
\theta_n^{(k)}\widehat X_n^{(k)}(\ell_n^{(i)}(u))
\longrightarrow\xi^{(k)}(\ell^{(i)}_u),
\end{equation}
whenever \(\xi^{(k)}\) is continuous at \(\ell^{(i)}_u\).

We show that the set of \(u\) for which
\(\xi^{(k)}\) is discontinuous at \(\ell^{(i)}_u\) is countable.
Define
\[
D_{k,i}:=
\{u\ge0 \mid \ell^{(i)}_u\in \disc(\xi^{(k)})\}.
\]
Since 
\[
\{u\ge0 \mid \ell^{(i)}_u=s\}
=[\xi^{(i)}(s-),\xi^{(i)}(s)],
\]
we have
\[
D_{k,i} = \bigcup_{s \in \disc(\xi^{(k)})}[\xi^{(i)}(s-),\xi^{(i)}(s)].
\]
Then the absence of common jump times of \(\xi^{(i)}\) and \(\xi^{(k)}\) implies
\[
D_{k,i}
=
\{\xi^{(i)}(s) \mid s\in\disc(\xi^{(k)})\},
\]
which is countable.
Hence, \eqref{Eq12} holds outside a countable
set, and
\[
F_n^{(i)}(u)\longrightarrow
F_i(u):=
 u+\sum_{\substack{k\in I_{0,\mathrm{dom}}\\k\ne i}}
 \xi^{(k)}(\ell^{(i)}_u)
\]
on a dense subset of \([0,\infty)\).
The function \(F_i\) starts at zero, is strictly increasing, and
is unbounded. By definition, its right-continuous inverse is
\(\Phi^{(i)}\).
Applying Lemma~\ref{lem:dense_pointwise_inverse} once more yields
\[
\widehat O_n^{(i)}=(F_n^{(i)})^{-1}
\longrightarrow\Phi^{(i)}
\quad\text{on} \quad C([0,\infty),\bR).
\]
Since \(\theta_n^{(i)}\to1\), it follows that
\begin{equation}\label{eq:occupation-dominant-coupled-limit}
\theta_n^{(i)}\widehat O_n^{(i)}
\longrightarrow\Phi^{(i)}
\quad\text{in }C([0,\infty),\mathbb R).
\end{equation}

\paragraph{Case 2: \(i\in I_{0,\mathrm{sub}}\).}
Fix \(i\in I_{0,\mathrm{sub}}\), and let
\(\tilde\ell^{(i)}:=(\tilde\xi^{(i)})^{-1}\).
The first inversion step is the same as in Case 1:
\[
\ell_n^{(i)}\longrightarrow\tilde\ell^{(i)} \quad
\text{local uniformly},
\]
by Lemma \ref{lem:dense_pointwise_inverse}.
For \(k\in(I_{0,\mathrm{sub}}\cup I_{0,\mathrm{res}})\setminus\{i\}\),
Lemma~\ref{lem:inverse_composition_pointwise} gives
\[
\theta_n^{(k)}\widehat X_n^{(k)}(\ell_n^{(i)}(u))
\longrightarrow0
\qquad(u\ge0).
\]
For \(k\in I_{0,\mathrm{dom}}\), it gives
\[
\theta_n^{(k)}\widehat X_n^{(k)}(\ell_n^{(i)}(u))
\longrightarrow\xi^{(k)}(\tilde\ell^{(i)}_u)
\]
whenever \(\xi^{(k)}\) is continuous at \(\tilde\ell^{(i)}_u\).
As in Case 1, the exceptional set
\[
\widetilde D_{k,i}
:=\{u \geq 0 \mid \tilde\ell^{(i)}_u\in\disc(\xi^{(k)})\}
=
\{\tilde\xi^{(i)}(s) \mid s\in\disc(\xi^{(k)})\}
\]
is countable, because \(\tilde\xi^{(i)}\) and \(\xi^{(k)}\) have no
common jump times.
Since \(\theta_n^{(i)}\to0\), we obtain
\begin{equation}\label{eq:occupation-subdominant-inverse-limit}
F_n^{(i)}(u)\longrightarrow
F_i(u):=S(\tilde\ell^{(i)}_u)
\end{equation}
on the dense complement of
\(\bigcup_{k\in I_{0,\mathrm{dom}}}\widetilde D_{k,i}\).

Unlike Case 1, this \(F_i\) is not strictly increasing.
We identify its right-continuous inverse explicitly.
For \(v\ge0\), put \(r=S^{-1}(v)\). By the definitions of the
right-continuous inverses,
\[
F_i^{-1}(v)=
\begin{cases}
\tilde\xi^{(i)}(r) &(S(r)=v),\\
\tilde\xi^{(i)}(r-) &(S(r)>v).
\end{cases}
\]
In the second case, \(S\) jumps at \(r\), so
\(\tilde\xi^{(i)}(r-)=\tilde\xi^{(i)}(r)\) by the absence of common
jump times.
Therefore,
\begin{equation}\label{eq:occupation-subdominant-inverse-identity}
F_i^{-1}(v)
=
\tilde\xi^{(i)}(S^{-1}(v))
=
\Phi^{(i)}(v)
\qquad(v\ge0).
\end{equation}
Also, \(F_i(0)=0\), \(F_i(\infty)=\infty\), and
\(F_i^{-1}(0)=0\) since \(F_i(u)>0\) for every \(u>0\).
Thus, Lemma \ref{lem:dense_pointwise_inverse}, applied to
\eqref{eq:occupation-subdominant-inverse-limit}, yields
\[
\widehat O_n^{(i)}(v)\longrightarrow\Phi^{(i)}(v)
\]
at every positive continuity point of \(\Phi^{(i)}\), and also
at \(v=0\). In particular, this convergence holds whenever
\(\tilde\xi^{(i)}\) is continuous at \(S^{-1}(v)\), because
\(S^{-1}\) is continuous.

It remains to verify this continuity at the fixed times \((t_{i})_{i=1}^{m}\).
Since \(S\) and \(\tilde\xi^{(i)}\) are independent and \(\tilde\xi^{(i)}\) has no fixed jump times, it follows
\[
\bP\left[
S^{-1}(t_k)\in\disc(\tilde\xi^{(i)})
\right]
=
\bE\left[\bP\left[
S^{-1}(t_k)\in\disc(\tilde\xi^{(i)})
~\middle|~S
\right]\right]=0
\qquad(1\le k\le m).
\]
After removing a finite union of null sets over \(i\) and \(k\),
we therefore have
\begin{equation}\label{eq:occupation-subdominant-coupled-limit}
\bigl(\widehat O_n^{(i)}(t_k)\bigr)_{k=1}^m
\longrightarrow
\bigl(\Phi^{(i)}(t_k)\bigr)_{k=1}^m
\quad\text{for every }i\in I_{0,\mathrm{sub}},
\quad\text{a.s.}
\end{equation}

\paragraph{Case 3: \(i\in I_{0,\mathrm{res}}\).}
Fix \(i\in I_{0,\mathrm{res}}\).
Since \(\widehat X_n^{(i)}\to0\) and \(S_n\to S\) in \(J_1\),
Lemma~\ref{lem:inverse_composition_pointwise} gives
\[
\widehat X_n^{(i)}(R_n(t))\longrightarrow0
\qquad(t\ge0).
\]
Together with \eqref{eq:occupation-coupled-upper-bound}, this implies
\[
\widehat O_n^{(i)}(t)\longrightarrow0
\qquad(t\ge0).
\]
By non-negativity and monotonicity, for every \(T>0\),
\[
\sup_{0\le t\le T}\widehat O_n^{(i)}(t)
=\widehat O_n^{(i)}(T)\longrightarrow0.
\]
Since \(\theta_n^{(i)}\to1\), we conclude that
\begin{equation}\label{eq:occupation-residual-coupled-limit}
\theta_n^{(i)}\widehat O_n^{(i)}
\longrightarrow0=\Phi^{(i)}
\quad\text{in }C([0,\infty),\mathbb R).
\end{equation}
This argument also covers identically zero coordinates.

Combining \eqref{eq:occupation-coupled-clock-inverse},
\eqref{eq:occupation-dominant-coupled-limit},
\eqref{eq:occupation-subdominant-coupled-limit}, and
\eqref{eq:occupation-residual-coupled-limit} on the same probability
space gives
\begin{align*}
&\left(
\bigl(\theta_n^{(i)}\widehat O_n^{(i)}\bigr)_
 {i\in I_{0,\mathrm{dom}}\cup I_{0,\mathrm{res}}},
R_n,
\left(\bigl(\widehat O_n^{(i)}(t_k)\bigr)_{k=1}^m\right)_
 {i\in I_{0,\mathrm{sub}}}
\right)
\\
&\qquad\longrightarrow
\left(
\bigl(\Phi^{(i)}\bigr)_{i\in I_{0,\mathrm{dom}}\cup I_{0,\mathrm{res}}},
S^{-1},
\left(\bigl(\Phi^{(i)}(t_k)\bigr)_{k=1}^m\right)_
 {i\in I_{0,\mathrm{sub}}}
\right)
\end{align*}
almost surely on
\[
C\left([0,\infty),
\mathbb R^{I_{0,\mathrm{dom}}\cup I_{0,\mathrm{res}}}\times\mathbb R\right)
\times(\mathbb R^m)^{I_{0,\mathrm{sub}}}.
\]
For each \(n\), the relation \eqref{eq:coupled-occupation-dictionary} identifies the law on the left with 
\[
\left(
\left(\frac{O^{\ast,(i)}(\lambda_n\cdot)}{\lambda_n}\right)_
 {i\in I_{0,\mathrm{dom}}\cup I_{0,\mathrm{res}}},
\frac{L_{A_{\lambda_n\cdot}}}{h(\lambda_n)},
\left(
\left(\frac{O^{\ast,(i)}(\lambda_n t_k)}{g_i(h(\lambda_n))}\right)_{k=1}^m
\right)_{i\in I_{0,\mathrm{sub}}}
\right).
\]
Hence, the asserted convergence in distribution holds along \(\lambda_n\).
Since the sequence \(\lambda_n\to\infty\) was arbitrary, the theorem follows.
\end{proof}

\begin{Rem} \label{rem:endpointCase}
In formulating the sub-dominant condition \eqref{assumption:regularVariation2}, we have restricted ourselves to the case \(\tilde{\alpha}_{i}\in[\alpha,1)\).
This condition can be extended to include the endpoint case
\(\tilde{\alpha}_{i}=1\) as follows.
For \(\tilde{\alpha}_{i}=1\), we assume that there exist a function
\(g_{i}(\lambda)\to\infty\) and a constant \(\tilde{\beta}_{i}>0\) such that
\[
\left(
\frac{\Gamma^{(i)}_{\lambda t}}{g_i(\lambda)}
\right)_{t\ge0}
\xrightarrow[\lambda\to\infty]{d}
(\tilde{\beta}_{i}t)_{t\ge0}
\]
on \(D([0,\infty),\bR)\).
The convergence of a scaled subordinator to a deterministic drift is known as
\textit{relative stability}, and various necessary and sufficient conditions have been established;
see, for example, \cite[Section 8.8]{Regularvariation}.
The most typical case is when \(\Gamma^{(i)}\) has a finite mean, namely
\[
\tilde{\beta}_{i}
=
ca_{f(0)}1_{\{i=0\}} + \int_{0}^{\infty}x \rho^{(i)}(dx)
<\infty.
\]
Then the strong law of large numbers applies, and one can simply take
\(g_i(\lambda)=\lambda\).

The proof of Theorem \ref{thm:occupation_limit} extends to this endpoint case
with no essential modification.
In fact, the argument becomes simpler:
the limit process is the deterministic drift \(\tilde{\xi}^{(i)}(t)=\tilde{\beta}_{i}t\), and both this process and its inverse are continuous and strictly increasing.
Thus, the complications caused by jump times disappear.
Consequently, for this component the convergence can be strengthened from finite-dimensional distributions to convergence on \(C([0,\infty),\bR)\).

The resulting limit process is \(\Phi^{(i)}(t)=\tilde{\beta}_{i}S^{-1}(t)\).
Since \(S^{-1}\) is the inverse of an \(\alpha\)-stable subordinator, the
one-dimensional distributions of \(\Phi^{(i)}\) are dilations of Mittag-Leffler distributions.
Thus, this endpoint case corresponds to the classical Darling--Kac-type limit random variable.
\end{Rem}

\begin{Rem} \label{rem:skewBessel}
The limit distribution \(\left((\Phi^{(i)})_{i\in I_{0,\mathrm{dom}}},S^{-1}\right)\) can be identified with the joint distribution of the occupation times and the
local time of a skew Bessel process.  We briefly recall this process and explain
the identification.

Let \(L_i\) \((i\in I_{0,\mathrm{dom}})\) be distinct half-lines in the plane
\(\bR^2\) with common endpoint \(0\), and set
\[
E:=\bigcup_{i\in I_{0,\mathrm{dom}}}L_i.
\]
Since \(\sum_{i\in I_{0,\mathrm{dom}}}\beta_i=1\) and \(\beta_i>0\) for
\(i\in I_{0,\mathrm{dom}}\), the vector \(\beta=(\beta_i)_{i\in I_{0,\mathrm{dom}}}\) is a probability vector with strictly positive entries.

Heuristically, the skew Bessel process \(X^{(\alpha,\beta)}\) on \(E\) is the Markov process which moves on each ray as a Bessel process of dimension
\(2-2\alpha\), and, whenever it returns to the common endpoint \(0\), chooses the ray of the next excursion according to the probability vector
\(\beta\).
This description is only heuristic, because \(0\) is regular for itself and the process leaves \(0\) instantaneously, meaning that the excursion measure is infinite.
A rigorous formulation is given by Itô's excursion theory; see, for example, \cite{MR3647066}.

More precisely, \(X^{(\alpha,\beta)}\) is the \(E\)-valued Markov process with no stagnancy at \(0\) whose excursion measure away from \(0\) is
\(n^{(\alpha,\beta)} = \sum_{i\in I_{0,\mathrm{dom}}}\beta_i n_i^{(\alpha)}\), where \(n_i^{(\alpha)}\) is the excursion measure of the Bessel process of dimension \(2-2\alpha\) on the ray \(L_i\).
We use the normalization of \(n_i^{(\alpha)}\) for which the lifetime process associated with \(n_i^{(\alpha)}\) has Laplace exponent \(q^\alpha\).
Then, under \(n^{(\alpha,\beta)}\), the cumulative excursion lifetime on \(L_i\) has Laplace exponent \(\beta_i q^\alpha\), and the total inverse local time has Laplace exponent \(q^\alpha\).
Such a process is also known as a Walsh process; see \cite{MR1022918}.

Let \(O^{(\alpha,\beta),(i)}(t)\) be the occupation time of \(X^{(\alpha,\beta)}\)
on the ray \(L_i\), and let \(L^{(\alpha,\beta)}_t\) be its local time at \(0\).
By Williams formula applied to \(X^{(\alpha,\beta)}\), the joint occupation
time and local time process satisfies
\[
\left(
\bigl(O^{(\alpha,\beta),(i)}(\cdot)\bigr)_{i\in I_{0,\mathrm{dom}}},
L^{(\alpha,\beta)}
\right)
\overset{d}{=}
\left(
\bigl(\Phi^{(i)}(\cdot)\bigr)_{i\in I_{0,\mathrm{dom}}},
S^{-1}
\right)
\]
on
\[
C([0,\infty),\bR^{I_{0,\mathrm{dom}}}\times\bR).
\]
Consequently, the dominant part of Theorem \ref{thm:occupation_limit}, together
with the local-time component, coincides in distribution with the occupation
time and local time process of a skew Bessel process.
\end{Rem}

\begin{Rem} \label{rem:limit_distribution}
The fixed-time distributions in Theorem \ref{thm:occupation_limit} can be
identified more explicitly.

\begin{enumerate}
\item \textbf{Dominant classes.}
For the dominant classes \(I_{0,\mathrm{dom}}\), Remark \ref{rem:skewBessel}
shows that
\[
\left(
\left(\Phi^{(i)}(\cdot)\right)_{i\in I_{0,\mathrm{dom}}},
S^{-1}
\right)
\]
has the same distribution as the joint occupation times and the local time at
\(0\) of the skew Bessel process \(X^{(\alpha,\beta)}\).
It is shown in
\cite[Theorem 1(ii)]{MR1022918} that, for each \(t>0\),
\[
\left(
\bigl(O^{(\alpha,\beta),(i)}(t)\bigr)_{i\in I_{0,\mathrm{dom}}},
L^{(\alpha,\beta)}(t)
\right)
\]
has the same distribution as
\[
\left(
\left(
\frac{t \beta_{i}^{1/\alpha} T_{i}}{\sum_{j\in I_{0,\mathrm{dom}}}\beta_{j}^{1/\alpha}T_{j}}
\right)_{i\in I_{0,\mathrm{dom}}},
\frac{t^\alpha}{\left(\sum_{j\in I_{0,\mathrm{dom}}}\beta_{j}^{1/\alpha}T_j\right)^\alpha}
\right),
\]
where \((T_i)_{i\in I_{0,\mathrm{dom}}}\) are i.i.d.\ positive
\(\alpha\)-stable random variables characterized by \(\bE[e^{-qT_i}] = \exp(-q^\alpha) \ (q\ge0)\).
Thus,
\[
\left(
\frac{\Phi^{(i)}(t)}{t}
\right)_{i\in I_{0,\mathrm{dom}}}  \overset{d}{=} \left(\frac{\beta_{i}^{1/\alpha} T_{i}}{\sum_{j\in I_{0,\mathrm{dom}}}\beta_{j}^{1/\alpha}T_{j}}\right)_{i \in I_{0,\mathrm{dom}}},
\]
and this distribution is a multidimensional generalization of Lamperti's generalized arcsine law (see \cite{Lamperti1958}).
Further explicit formulae, including double Laplace transforms, are given in
\cite[Theorem 4]{MR1022918}.

\item \textbf{Sub-dominant classes.}
For \(i\in I_{0,\mathrm{sub}}\), the one-dimensional marginal law of
\(\Phi^{(i)}(t)\) has a simple Mittag--Leffler type Laplace transform.  Since
\[
\Phi^{(i)}(t)=\tilde{\xi}^{(i)}(S^{-1}_t)
\]
and \(S\) is an \(\alpha\)-stable subordinator with Laplace exponent \(q^\alpha\),
we obtain, for \(q\ge0\),
\[
\bE\left[e^{-q\Phi^{(i)}(t)}\right]
=
\bE\left[
\exp\left(-\tilde{\beta}_{i}q^{\tilde{\alpha}_{i}}S^{-1}_{t}\right)
\right]
=
E_\alpha\left(-\tilde\beta_i q^{\tilde\alpha_i}t^\alpha\right),
\]
where
\[
E_\alpha(z):=\sum_{m=0}^\infty \frac{z^m}{\Gamma(1+\alpha m)}
\]
is the Mittag--Leffler function.
\end{enumerate}
\end{Rem}

\subsection{Scaling limits for waiting times}

The occupation-time limits in Theorem \ref{thm:occupation_limit}
describe the cumulative amounts of time spent in the different
excursion classes and at \(o\).
We now consider the last and next visits to \(o\), which describe
the location of the excursion interval containing an observation
time.
The following result gives their scaling limit jointly with the
occupation-time and inverse-clock limits of that theorem.
Throughout this subsection, we retain its assumptions and notation.

Define
\[
G_t^\ast:=\sup\{0\le s\le t\mid X_s^\ast=o\},
\quad
D_t^\ast:=\inf\{s>t\mid X_s^\ast=o\}
\quad(t\ge0).
\]
These are the last visit time to \(o\) at or before \(t\) and the
next visit time after \(t\), understood in the supremum and infimum
senses, respectively.
The associated elapsed and residual waiting times are
\(t-G_t^\ast\) and \(D_t^\ast-t\).
In particular, when \(t\) lies in the interior of an excursion
interval away from \(o\), \(G_t^\ast\) and \(D_t^\ast\) are its left
and right endpoints, so these differences give the elapsed and
remaining durations of that excursion.
The definition allows \(D_t^\ast=t\) when visits to \(o\) occur
arbitrarily close to \(t\) from the right.

To describe the limit, it is convenient to use the last and next
points of a closed set.
For an unbounded closed set \(F\subset[0,\infty)\) containing zero,
write
\[
G(F)_t:=\sup(F\cap[0,t]),
\quad
D(F)_t:=\inf(F\cap(t,\infty))
\quad(t\ge0).
\]
We distinguish the origin and excursion components of the limiting
clock by setting
\[
S^{(0)}:=
\begin{cases}
\xi^{(0)} &(0\in I_{0,\mathrm{dom}}),\\
0 &(0\notin I_{0,\mathrm{dom}}),
\end{cases}
\quad
S^{\mathrm{exc}}:=
\sum_{i\in I_{0,\mathrm{dom}}\setminus\{0\}}\xi^{(i)},
\]
so that \(S=S^{(0)}+S^{\mathrm{exc}}\).
Here the stable subordinators \(\xi^{(i)}\) are the same as those
used in Theorem \ref{thm:occupation_limit}.
Only the jump intervals associated with \(S^{\mathrm{exc}}\)
represent excursions away from \(o\); those associated with
\(S^{(0)}\) represent time spent at \(o\).
Accordingly, define the closed set
\begin{equation}\label{eq:waiting-limit-zero-set}
\cZ:=[0,\infty)\setminus
\bigcup_{\substack{s>0\\\Delta S^{\mathrm{exc}}(s)>0}}
(S(s-),S(s)).
\end{equation}
Writing \(\cR_S:=\overline{\{S(s)\mid s\ge0\}}\), the absence of
common jump times of \(S^{(0)}\) and \(S^{\mathrm{exc}}\) gives
\[
\cZ
=
\cR_S\cup
\bigcup_{\substack{s>0\\\Delta S^{(0)}(s)>0}}
[S(s-),S(s)].
\]
Thus, \(\cZ\) is obtained by filling in those gaps of \(\cR_S\)
associated with \(S^{(0)}\), while retaining the gaps associated
with excursions away from \(o\).
The last and next points of \(\cZ\) describe the waiting-time limit
in the following theorem.

\begin{Thm}\label{thm:waiting-times}
Under the assumptions of Theorem \ref{thm:occupation_limit}, and with the notation above,
\begin{equation}\label{eq:waiting-times-limit}
\left(
\frac{G^\ast_{\lambda\cdot}}{\lambda},
\frac{D^\ast_{\lambda\cdot}}{\lambda}
\right)
\xrightarrow[\lambda\to\infty]{d}
\bigl(G(\cZ),D(\cZ)\bigr)
\quad\text{in }D([0,\infty),\bR^2)
\end{equation}
with the \(J_1\)-topology.
This convergence holds jointly with that in Theorem \ref{thm:occupation_limit}.
\end{Thm}

\begin{Rem}\label{rem:waiting-times}
When \(\beta_0=0\), we have \(S^{(0)}\equiv0\) and
\(\cZ=\cR_S\), recovering the usual waiting-time limits associated
with the closed range of an \(\alpha\)-stable subordinator.
When \(\beta_0=1\), we have \(S^{\mathrm{exc}}\equiv0\) and
\(\cZ=[0,\infty)\), so both limiting waiting-time processes are the
identity function.
\end{Rem}

\begin{proof}[Proof of Theorem \ref{thm:waiting-times}]
We follow the waiting-time argument of \cite{Lamperti1962}, combined
with the coupling approach of \cite{Sera2020}.
Write
\[
\cT:=T\circ\eta=\Gamma^{(0)}+\Gamma.
\]
The extra point to note is to distinguish whether a jump of
\(\mathcal T\) originates from \(\Gamma^{(0)}\) or from \(\Gamma\).
By Proposition \ref{prop:Gamma0LawAndIndependence}, these two
components have no common jump times almost surely.

First, from Proposition \ref{prop:excursionIntervalCorrespondence} and Corollary \ref{cor:excursionIntervalCorrespondence} we have
\begin{equation}\label{eq:waiting-original-zero-set}
\cZ^\ast:=\overline{\{u\ge0 \mid X_u^\ast=o\}}
=
[0,\infty)\setminus
\bigcup_{\substack{s>0 \\\Delta\Gamma(s)>0}}
(\mathcal T_{s-},\mathcal T_s).
\end{equation}
Note that the removed intervals are those generated by jumps of \(\Gamma\), not \(\Gamma^{(0)}\).
By definition, we have
\[
G^\ast=G(\cZ^\ast),
\quad D^\ast=D(\cZ^\ast)
\quad\text{a.s.}
\]

Fix an arbitrary sequence \(\lambda_n\to\infty\), and use the
coupling and notation from the proof of Theorem \ref{thm:occupation_limit}.
Define
\[
Q_n^{(0)}:=\theta_n^{(0)}\widehat X_n^{(0)},
\quad
Q_n^{\mathrm{exc}}:=
\sum_{i\in I}\theta_n^{(i)}\widehat X_n^{(i)},
\quad S_n=Q_n^{(0)}+Q_n^{\mathrm{exc}}.
\]
Then
\begin{equation}\label{eq:waiting-two-clock-convergence}
(Q_n^{(0)},Q_n^{\mathrm{exc}})
\longrightarrow(S^{(0)},S^{\mathrm{exc}})
\quad\text{a.s. in }D([0,\infty),\bR^2).
\end{equation}
Set
\[
\mathcal Z_n:=[0,\infty)\setminus
\bigcup_{\substack{s>0 \\\Delta Q_n^{\mathrm{exc}}(s)>0}}
(S_n(s-),S_n(s)),
\quad
\widehat G_n:=G(\mathcal Z_n),
\quad
\widehat D_n:=D(\mathcal Z_n).
\]
For every fixed \(n\), \eqref{eq:waiting-original-zero-set} gives
\begin{equation}\label{eq:waiting-coupled-law}
(\widehat G_n,\widehat D_n)
\overset d=
\left(
\frac{G^\ast_{\lambda_n\cdot}}{\lambda_n},
\frac{D^\ast_{\lambda_n\cdot}}{\lambda_n}
\right).
\end{equation}
This identity holds jointly with all the coupled coordinates in
\eqref{eq:coupled-occupation-dictionary}, since \(\widehat{G}_{n}\) and \(\widehat{D}_{n}\) are measurably defined by \((\widehat{X}_{n}^{(i)})_{i \in I_{0}}\).
We work on a probability-one event on which
\eqref{eq:waiting-two-clock-convergence} holds, all the total clocks
\(S_n,S\) start at zero and are strictly increasing and unbounded,
and the two components of each clock have no common jump times.

Write
\[
G:=G(\cZ),\quad D:=D(\cZ).
\]
All these waiting-time processes are non-decreasing and c\`adl\`ag.
We first establish that 
\begin{align}
\widehat{G}_{n} \xrightarrow[n\to\infty]{d}
 G \quad \text{and} \quad \widehat{D}_{n} \xrightarrow[n\to\infty]{d}
 D \quad \text{in }D([0,\infty),\bR). \label{Eq16}
\end{align}

By \cite[Theorem]{Stone1963}, the convergence \eqref{Eq16} holds when the following two conditions are satisfied:
\begin{enumerate}
\item For every \(\eps,N > 0\)
\[
\lim_{c \downarrow 0}\limsup_{n \to \infty}\bP [\Delta_{N}(c,\widehat{G}_{n}) > \eps] = \lim_{c \downarrow 0}\limsup_{n \to \infty}\bP [\Delta_{N}(c,\widehat{D}_{n}) > \eps] = 0,
\]
where
\[
\Delta_{N}(c,f) := \sup_{\substack{t-c < t_{1} < t < t_{2} < t+c \\ 0 \leq t_{1} < t_{2} \leq N}} (|f(t_{1}) - f(t)| \wedge |f(t) - f(t_{2})|) \quad (c > 0, \ f \in D([0,\infty),\bR)).
\]
\item Every finite-dimensional distribution of \(\widehat{G}_{n}\) and \(\widehat{D}_{n}\) weakly converges to that of \(G\) and \(D\), respectively.
\end{enumerate}
From the argument in \cite[p.689]{Lamperti1962} or simply drawing the graph of \(t \mapsto G(F)_{t}\) and \(t \mapsto D(F)_{t}\), we see that for every \(0 < c < N\) 
\[
\Delta_{N}(c,\widehat{G}_{n}) \leq 2c,\quad  \Delta_{N}(c,\widehat{D}_{n}) \leq 2c.
\]
Thus, condition (i) holds, and it remains to show (ii).

We first show pointwise convergence of \(\widehat{G}_{n}\) and \(\widehat{D}_{n}\) to \(G\) and \(D\), respectively, at continuity points of the limits.
Let \(\kappa_{n} \in \Lambda\) be time-change functions for the \(J_{1}\)-convergence of \eqref{eq:waiting-two-clock-convergence}.
For a jump time \(s>0\) of \(S\), put \(s_n:=\kappa_n(s)\).
Then the \(J_{1}\)-convergence gives
\[
S_n(s_n-)\longrightarrow S(s-),
\quad S_n(s_n)\longrightarrow S(s),
\]
and
\[
\bigl(\Delta Q_n^{(0)}(s_n),\Delta Q_n^{\mathrm{exc}}(s_n)\bigr)
\longrightarrow
\bigl(\Delta S^{(0)}(s),\Delta S^{\mathrm{exc}}(s)\bigr).
\]
If \(\Delta S^{\mathrm{exc}}(s)>0\), then
\(\Delta Q_n^{\mathrm{exc}}(s_n)>0\) for all sufficiently large
\(n\).
If \(\Delta S^{(0)}(s)>0\), then
\(\Delta Q_n^{(0)}(s_n)>0\) and hence
\(\Delta Q_n^{\mathrm{exc}}(s_n)=0\) for all sufficiently large
\(n\).
Thus, the corresponding jump interval is respectively removed or
filled in for both the approximating and limiting zero sets.
Then for fixed \(t\in(S(s-),S(s))\), we have for all sufficiently large \(n\),
\[
(\widehat G_n(t),\widehat D_n(t))=
\begin{cases}
(S_n(s_n-),S_n(s_n)) &(\Delta S^{\mathrm{exc}}(s)>0),\\
(t,t) &(\Delta S^{(0)}(s)>0).
\end{cases}
\]
Consequently,
\[
(\widehat G_n(t),\widehat D_n(t))\longrightarrow(G(t),D(t))
\]
for 
\[
t \in \bigcup_{\substack{s > 0 \\ \Delta S(s) > 0}}(S(s-),S(s)).
\]
Since \(S\) is driftless, these interiors form a dense subset of
\([0,\infty)\).
Moreover, \(\widehat G_n(0)=\widehat D_n(0)=G(0)=D(0)=0\).
By monotonicity, we therefore obtain
\begin{align}
\begin{aligned}
\widehat G_n(t)&\longrightarrow G(t)
\quad \text{at every continuity point of }G,\\
\widehat D_n(t)&\longrightarrow D(t)
\quad \text{at every continuity point of }D.  
\end{aligned}
\label{eq:waiting-pointwise}
\end{align}

To deduce condition (ii) from \eqref{eq:waiting-pointwise}, we show that every fixed time is almost surely a continuity point of both \(G\) and \(D\).
By definition, we see that the discontinuity sets are
\begin{equation}\label{eq:waiting-discontinuities}
\disc(G)=\{S(s) \mid s>0,\ \Delta S^{\mathrm{exc}}(s)>0\},
\quad
\disc(D)=\{S(s-) \mid s>0,\ \Delta S^{\mathrm{exc}}(s)>0\}.
\end{equation}
That is, a deterministic time \(a>0\) can be a discontinuity point of \(G\) or \(D\) only if it is an endpoint of a jump interval of the range of \(S\).
This has probability zero.
Indeed, by the compensation formula and the non-atomicity of the L\'evy measure \(\nu\) of \(S^{\mathrm{exc}}\),
\[
\bP[a \in \disc(G)] \leq \bE\left[
\sum_{\substack{s>0 \\ \Delta S^{\mathrm{exc}}(s) > 0 }}1_{\{S(s-)+\Delta S(s)=a\}}
\right]
=
\bE\left[
\int_0^\infty\int_{(0,\infty)}
1_{\{S(s-)+x=a\}}\,\nu(dx)\,ds
\right]
=0.
\]
For the left endpoints, fix \(\varepsilon>0\).
The compensation formula gives
\[
\bE\left[
\sum_{s>0}
1_{\{\Delta S^{\mathrm{exc}}(s)>\varepsilon\}}
1_{\{S(s-)=a\}}
\right]
=
\nu((\varepsilon,\infty))
\bE\left[
\int_0^\infty 1_{\{S(s-)=a\}}\,ds
\right]
=0.
\]
Indeed, strict increase of \(S\) implies that \(\{s>0 \mid S(s-)=a\}\) contains at most one point.
Letting \(\varepsilon\downarrow0\), we obtain \(\bP[a\in\disc(D)]=0\).
Thus, \(G\) and \(D\) have no fixed discontinuities.
For every finite sequence of deterministic times
\(0\le t_1<\cdots<t_m\), the pointwise convergence above therefore
holds simultaneously almost surely:
\begin{equation}\label{eq:waiting-finite-dimensional}
\bigl(\widehat G_n(t_k),\widehat D_n(t_k)\bigr)_{k=1}^m
\longrightarrow
\bigl(G(t_k),D(t_k)\bigr)_{k=1}^m
\quad\text{a.s.}
\end{equation}
This proves condition (ii).
Consequently, the convergence \eqref{Eq16} holds.

Finally, we combine these convergences with the joint limit in Theorem \ref{thm:occupation_limit}.
Fix the sequence of deterministic times for the sub-dominant classes in Theorem \ref{thm:occupation_limit}.
Let \(V_n,V\) be its coupled sequence and limit, and let \(\cE\)
denote the space in which \(V_{n}\) and \(V\) take values.
The proof of the theorem gives \(V_n\to V\) almost surely in \(\cE\).

By this convergence and \eqref{Eq16}, the laws of \((V_n,\widehat G_n,\widehat D_n)\) are tight on
\[
\cE\times D([0,\infty),\bR)\times D([0,\infty),\bR).
\]
Moreover, for every finite sequence of times \(0\le u_1<\cdots<u_r\), convergence \eqref{eq:waiting-finite-dimensional} holds jointly with \(V_n\to V\):
\[
\left(
V_n,
\bigl(\widehat G_n(u_k),\widehat D_n(u_k)\bigr)_{k=1}^{r}
\right)
\longrightarrow
\left(
V,
\bigl(G(u_k),D(u_k)\bigr)_{k=1}^{r}
\right)
\quad\text{a.s.}
\]
Since \(G\) and \(D\) have no fixed discontinuities, these joint laws identify the limit.
Thus,
\[
(V_n,\widehat G_n,\widehat D_n)
\xrightarrow{d}(V,G,D)
\]
on the above product space.

By \eqref{eq:waiting-discontinuities} and strict increase of \(S\),
the processes \(G\) and \(D\) have no common jump times almost surely.
Lemma \ref{lem:coordinatewise_to_multivariate_J1} and the continuous
mapping theorem therefore allow us to replace the last two scalar
factors by \(D([0,\infty),\bR^2)\).
Consequently,
\[
(V_n,(\widehat G_n,\widehat D_n))
\xrightarrow{d}(V,(G,D))
\quad\text{in }\cE\times D([0,\infty),\bR^2).
\]
Returning to the original joint laws by
\eqref{eq:waiting-coupled-law} and
\eqref{eq:coupled-occupation-dictionary}, and recalling that
\(\lambda_n\to\infty\) was arbitrary, proves
\eqref{eq:waiting-times-limit} and its joint version.
\end{proof}

\section{Examples} \label{section:example}

In this section, we illustrate how the preceding limit theorems apply to
concrete processes.
We focus on examples in the state-space partition setting,
where the preservation of the excursion classes under the time change is
immediate.

We first consider continuous-time Markov chains on discrete rays.
This example shows how regular variation of excursion lifetime tails can be verified from hitting-time asymptotics on individual rays.
We then discuss one-dimensional diffusions, where the corresponding tails can be determined from the speed measure and its atom at the distinguished point determines the stagnancy rate.

\subsection{Continuous-time Markov chains on discrete rays}

Let \(n\in \bN\), and define \(S \subset \bZ^{2}\) by
\[
S=\{o\}\sqcup \bigsqcup_{1\le i\le n}S_i,
\quad
o := (0,0), \quad S_{0} := \{o\}, \quad S_i:=\{(i,x) \in \bZ^{2} \mid x\geq 1\} \quad (1 \leq i\leq n).
\]
We call each \(S_i \ (1 \leq i \leq n)\) a ray and regard \(o\) as the common endpoint of the rays.

Let \(X = (X_{t},\bP_{(i,x)})\) be a non-explosive continuous-time Markov chain on \(S\) with conservative \(Q\)-matrix \(Q\).
We assume that \(X\) is recurrent at \(o\) and that every state has a holding time:
\[
q(z):=-q(z,z)<\infty \quad (z\in S).
\]
We further assume that the process can move from one ray to another only through \(o\).
More precisely, for \((i,x)\in S_{i}\) and
\((j,y)\in S_{j}\) with \(1 \leq i,j \leq n\) and \(i \neq j\), we assume
\[
q((i,x),(j,y))=0.
\]
Set
\[
r_{i} := \sum_{y \geq 1} q(o,(i,y)) \quad (1\le i\le n),
\]
and assume \(r_{i} > 0\) for every \(i = 1,\dots, n\) so that every ray is visited with positive probability.

We define the entrance distribution \(\mu_{i}\) of \(X\) from \(o\) into \(S_{i}\) by
\[
\mu_{i}((i,x)) := \frac{q(o,(i,x))}{r_{i}}
\quad (1\le i\le n,\ x \geq 1).
\]

For a path \(w\) on \(S\), let
\[
\tau_o(w):=\inf\{t>0\mid w_t=o\},
\]
and write \(\tau_o:=\tau_o(X)\) when no argument is displayed.
Define an excursion measure \(n\) of \(X\) by
\[
n(A)
:=
q(o)\bP_o\left[
(\theta_\sigma X)_{\cdot\wedge\tau_o(\theta_\sigma X)}
\in A
\right]
\quad(A\in\cB(D_0)),
\]
where
\[
\sigma:=\inf\{t>0\mid X_t\ne o\}.
\]
From the first step analysis, we can write as
\[
n(A) = \sum_{1 \leq i \leq n}r_{i}\bP_{\mu_{i}}[X_{\cdot \wedge \tau_{o}} \in A].
\]
Under this normalization of \(n\), the local time at \(o\) is given by
\[
L_t:=\int_{0}^{t} 1_{\{X_s=o\}}ds.
\]
The corresponding stagnancy rate is \(c=1\).

For \(i=1,\dots,n\), define
\[
D_0^{(i)} := \left\{ e\in D_0 \mid \zeta(e) > 0,\  e_{t}\in S_{i} \text{ for all }0<t<\zeta(e) \right\}.
\]
These sets are measurable, and since the chain cannot move between distinct
rays without hitting \(o\), we have
\[
n\left(D_0\setminus\bigcup_{i=1}^nD_0^{(i)} \right)=0,
\quad
D_0^{(i)}\cap D_0^{(j)} = \emptyset \quad (i\ne j).
\]
Thus, \((D_0^{(i)})_{1\le i\le n}\) corresponds to a state-space partition setting.

Let \(n^{(i)}:=n(\cdot \cap D_0^{(i)})\).
Then, for \(t\ge0\),
\[
n^{(i)}(\zeta>t)
=
r_{i}\bP_{\mu_{i}}(\tau_{o} > t).
\]
Hence, the excursion lifetime tail on each ray is determined by the tail of the hitting time \(\tau_o\) under \(\bP_{\mu_i}\).

We now incorporate the effect of the time change.
Suppose that the subordinator assigned to \(S_i \ (0 \leq i \leq n)\) is a stable subordinator with Laplace exponent
\[
\psi^{(i)}(q)=b_{i} q^{\gamma_{i}},
\quad (b_{i}>0, \ 0 < \gamma_{i} < 1).
\]

Assume that the original lifetime tail on the ray \(S_i \ (1 \leq i \leq n)\) is regularly varying:
\begin{align}
n^{(i)}(\zeta>t)
=
r_{i}\bP_{\mu_i}(\tau_{o} > t)
\sim
a_{i} t^{-\delta_i}K_{i}(t)
\quad (t\to\infty), \label{Eq18}
\end{align}
where \(a_{i}>0\), \(0 < \delta_{i} < 1\), and \(K_{i}\) is slowly varying at infinity.
Then the Laplace exponent \(\Psi^{(i)}\) of the original cumulative lifetime process satisfies
\[
\Psi^{(i)}(q) \sim \Gamma(1-\delta_i)a_i q^{\delta_i}K_i(1/q)
\quad (q \to 0).
\]
By Proposition \ref{prop:LEofTimeChange},
\[
\Psi^{\ast,(i)}(q)=\Psi^{(i)}(\psi^{(i)}(q)).
\]
Hence,
\[
\Psi^{\ast,(i)}(q) \sim \Gamma(1-\delta_i)a_i b_i^{\delta_i}
q^{\delta_i\gamma_i}K_i(q^{-\gamma_i})
\quad (q \to 0).
\]
Applying Karamata's Tauberian theorem, we obtain
\[
\rho^{(i)}((t,\infty))
\sim
a_i^{\ast} t^{-\delta_i\gamma_i}K_i(t^{\gamma_i}) \quad (t\to \infty, \ 1 \leq i \leq n),
\]
where \(\rho^{(i)}\) is the measure defined in \eqref{def:measureRho} and 
\[
a_i^{\ast}
:=
\frac{\Gamma(1-\delta_i)}{\Gamma(1-\delta_i\gamma_i)}
a_i b_i^{\delta_i}.
\]
Thus, the regular variation index of the lifetime tail changes
from \(-\delta_i\) to \(-\alpha_{i}\), where
\[
\alpha_i:=\delta_i\gamma_i\quad(1\le i\le n).
\]
For the origin component, \(c=1\) and Proposition \ref{prop:Gamma0LawAndIndependence} give
\[
\rho^{(0)}((t,\infty))=a_0^\ast t^{-\alpha_0}
\quad(t>0),
\quad
 a_0^\ast:=\frac{b_0}{\Gamma(1-\gamma_0)},
\quad\alpha_0:=\gamma_0.
\]
Set
\[
\alpha:=\min_{0\le i\le n}\alpha_i,
\quad
I_{\mathrm{min}}:=\{i\in\{0,\ldots,n\}\mid\alpha_i=\alpha\}.
\]
The indices in \(I_{\mathrm{min}}\) are the candidates for the dominant class
after the time change.
Define
\[
K(t):=\sum_{j\in I_{\mathrm{min}}}a_j^\ast K_j(t^{\gamma_j}),
\]
where \(K_0\equiv1\).
Then \(K\) is slowly varying.
We assume that, for every \(i\in I_{\mathrm{min}}\), the limit
\[
\beta_i:=
\lim_{t\to\infty}
\frac{a_i^\ast K_i(t^{\gamma_i})}{K(t)}
\]
exists, and set \(\beta_i:=0\) for \(i\notin I_{\mathrm{min}}\).
Then \(\beta_i\ge0\) and
\[
\sum_{i=0}^n\beta_i=1.
\]
For every \(0\le i\le n\), we have
\[
\frac{\rho^{(i)}((t,\infty))}
     {t^{-\alpha}K(t)}
\longrightarrow\beta_i \quad (t \to \infty).
\]
Hence, the dominant regular variation condition of Theorem \ref{thm:occupation_limit} is satisfied, with
\[
I_{0,\mathrm{dom}}:=\{i\in I_{\mathrm{min}}\mid\beta_i>0\}.
\]
Every non-dominant component is sub-dominant, since its L\'evy measure tail is regularly varying with a positive coefficient and an index in \([\alpha,1)\).
Thus, \(I_{0,\mathrm{res}}=\emptyset\) in this example.

\begin{Rem}
Fix \(i\in\{1,\ldots,n\}\).
In place of \eqref{Eq18}, suppose that the original cumulative lifetime process is relatively stable, as in Remark \ref{rem:endpointCase}.

Suppose that there exist a function \(f_i(\lambda)\to\infty\) and a constant \(m_i>0\) such that
\[
\frac{1}{f_i(\lambda)}\int_{{[0,\lambda t]\times D_{0}}}\zeta(e)N^{(i)}(ds\,de)
\xrightarrow[\lambda\to\infty]{d}
m_i t
\]
in \(D([0,\infty),\bR)\). 
Let \(k_i\) be an asymptotic inverse of \(f_i\), and write \(\ell_i(t):= k_i(t)/t\).
Then \(\ell_{i}\) is slowly varying, and the above convergence is equivalent to
\[
\Psi^{(i)}(q) \sim \frac{m_iq}{\ell_i(1/q)} \quad (q \to 0).
\]
This typically holds when
\[
m_i=\int_{D_0}\zeta(e)\,n^{(i)}(de)<\infty,
\]
in which case one can take \(f_i(\lambda)=\lambda\) and hence
\(\ell_{i}\equiv1\).

Assume that the subordinator assigned to the ray \(S_i\) satisfies
\[
\psi^{(i)}(q)\sim b_iq^{\gamma_i}L_i(1/q) \quad (b_i>0,\ 0<\gamma_i<1),
\]
where \(L_i\) is slowly varying.
Then
\[
\Psi^{\ast,(i)}(q) = \Psi^{(i)}(\psi^{(i)}(q)) \sim m_{i}b_{i}q^{\gamma_{i}} K_i(1/q),
\]
where \(K_i\) is the slowly varying function given by
\[
K_i(t) := \frac{L_i(t)}{\ell_i\left(t^{\gamma_i}/L_i(t)\right)}.
\]
Consequently,
\[
\rho^{(i)}((t,\infty)) \sim \frac{m_ib_i}{\Gamma(1-\gamma_i)} t^{-\gamma_i} K_i(t) \quad(t\to\infty).
\]
Thus, relatively stable cases can be treated in the same way.
In particular, even when the original excursion lifetime is light-tailed, for example exponentially decaying, the time-changed excursion lifetime may be regularly varying.
\end{Rem}

\subsection{One-dimensional diffusions}

We examine how our limit theorems apply to one-dimensional diffusions.
Occupation-time limit theorems for diffusions have been studied previously; see, for example, \cite{Kasahara1975,Watanabe1995,MR3647066}.

We briefly recall the general theory of one-dimensional diffusions (see e.g., \cite{MR345224}, \citet[Chapter 33]{Kallenberg-third} for details).
Let \(I \subset \bR\) be an interval with non-empty interior \(I^{\circ}\).
We consider a one-dimensional diffusion \(X = (X_{t},\bP_{x})\) on \(I\) that is non-singular and recurrent;
\[
\bP_{x}[\tau_{y} < \infty] = 1 \quad (x,y \in I),
\]
where \(\tau_y:=\inf\{t>0 \mid X_t=y\}\)
This assumption implies that when an endpoint of \(I\) is in \(I\), it is not a trap.
As is well known, the behavior of \(X\) on \(I^{\circ}\) is characterized by the scale function and speed measure.
That is, there exist a Radon measure \(m\) on \(I\) with full support and a strictly increasing continuous function \(s:I \to \bR\), and the local generator of \(X\) is represented as \(\frac{d}{dm}\frac{d}{ds}\).
By considering the process \(s(X)\) on \(s(I)\) instead of \(X\), we may always assume without loss of generality that \(X\) is in the natural scale; \(s(x) = x\).
We may also assume \(0 \in I^{\circ}\).

The process \(X\) admits an occupation density \(L^{x} = (L^{x}_{t})_{t \geq 0} \ (x \in I)\) with respect to the speed measure \(m\):
\[
\int_{0}^{t}f(X_{s})ds = \int_{I}f(y)L^{y}_{t}m(dy) \quad \bP_{x}\text{-a.s.}
\]
for every \(x\in I\), \(t\ge0\), and every non-negative measurable function \(f\).
Since \(X\) is a strong Markov process, it is regenerative at \(0\).
It is a standard fact that the local time \(L\) at \(0\) in our excursion framework coincides with \(L^{0}\) up to a multiplicative constant.
We normalize \(L\) so that \(L \equiv L^{0}\).
We denote the excursion measure of \(X\) away from \(0\) by \(n\).

Since \(X\) has continuous paths, every excursion away from \(0\) is either
positive or negative.  Define
\[
D_0^{(+)} = \{e\in D_0 \mid e_t>0\text{ for all }0<t<\zeta(e)\},
\]
\[
D_0^{(-)} = \{e\in D_0 \mid e_t<0\text{ for all }0<t<\zeta(e)\}.
\]
Then
\[
n\bigl(D_0\setminus(D_0^{(+)}\cup D_0^{(-)})\bigr)=0,
\quad
D_0^{(+)}\cap D_0^{(-)}=\emptyset.
\]
Thus, \((D_0^{(+)},D_0^{(-)})\) is a measurable partition of the excursion
space.
It is preserved by the time changes considered in this paper.
Set
\[
n^{(+)}:=n(\cdot\cap D_0^{(+)}),\quad
n^{(-)}:=n(\cdot\cap D_0^{(-)}).
\]
They can be identified, up to normalization of local time, with the excursion measures associated with the half-line diffusions obtained from the restrictions of the speed measure to the positive and negative half-lines.

With the above normalization, the inverse local time \(\eta\) at \(0\) has the representation
\[
\eta_s = m(\{0\})s + \int_{[0,s]\times D_0}\zeta(e) N(du\,de).
\]
Equivalently, its Laplace exponent is
\[
-\log \bE_0[e^{-q\eta_1}]
=
m(\{0\})q
+
\int_0^\infty (1-e^{-qt})
\bigl(n^{(+)}+n^{(-)}\bigr)(\zeta\in dt).
\]
Thus, in this normalization, the stagnancy rate at \(0\) is \(m(\{0\})\).

The regular variation of \(n^{(+)}(\zeta>t)\) and \(n^{(-)}(\zeta>t)\) can be verified from the asymptotic behavior of the speed measure on the corresponding half-line.
For example, suppose \(I = \bR\) and write
\[
M_+(x):=m((0,x]), \quad M_-(x):=m([-x,0)).
\]
For \(0<\alpha_\pm<1\), regular variation of \(M_\pm\) at infinity
with index \(1/\alpha_\pm-1\) yields regular variation of
\(n^{(\pm)}(\zeta>t)\) with index \(-\alpha_\pm\).
This follows from \cite[Theorem 2]{Kasahara1975}, together with the
Karamata's Tauberian theorem.
We can use this result as a criterion ensuring the regular variation assumptions in Theorem \ref{thm:occupation_limit}.

We now illustrate how to apply this criterion in a simple example.
Fix \(\alpha_{\pm}\in(0,1)\) and \(c\ge0\).  Let \(m\) be the Radon measure on
\(\bR\) given by
\[
\begin{aligned}
m(dx)
&=
2(2\alpha_-)^{1/\alpha_- -2}|x|^{1/\alpha_- -2}
1_{\{x<0\}}\,dx
+
c\delta_0(dx)  \\
&\quad+
2(2\alpha_+)^{1/\alpha_+ -2}x^{1/\alpha_+ -2}
1_{\{x>0\}}\,dx .
\end{aligned}
\]
Let \(X\) be the natural-scale diffusion on \(\bR\) with speed measure
\(m\), started at \(0\).
This diffusion is non-singular and null-recurrent.
The process \(X\) behaves as a natural-scale Bessel process of dimension parameter \(2-2\alpha_{+}\) (resp. \(2-2\alpha_{-}\)) on \((0,\infty)\) (resp. \((-\infty,0)\)) and has stagnancy at \(0\) determined by the atom \(c\delta_0\).
For this speed measure, the cumulative masses on the two half-lines
are
\[
M_\pm(x) = \frac{(2\alpha_\pm)^{1/\alpha_\pm-1}}{1-\alpha_\pm}
x^{1/\alpha_\pm-1}
\quad(x>0).
\]
The preceding criterion therefore gives
\[
n^{(\pm)}(\zeta>t) \sim a_\pm t^{-\alpha_\pm} \quad(t\to\infty)
\]
for some constants \(a_\pm>0\).

We assign stable subordinators to the positive excursions,
negative excursions, and the origin, with Laplace exponents
\[
\psi^{(i)}(q)=b_iq^{\gamma_i}
\quad
(b_i>0,\ 0<\gamma_i<1,\ i\in\{+,-,0\}).
\]
By Proposition~\ref{prop:LEofTimeChange} and Karamata's Tauberian
theorem,
\[
\rho^{(\pm)}((t,\infty))
=
n^{\ast,(\pm)}(\zeta>t)
\sim
a_\pm^\ast t^{-\alpha_\pm^\ast}
\quad(t\to\infty),
\]
where
\[
\alpha_\pm^\ast:=\alpha_\pm\gamma_\pm,
\quad
a_\pm^\ast
:=
\frac{\Gamma(1-\alpha_\pm)}
     {\Gamma(1-\alpha_\pm\gamma_\pm)}
a_\pm b_\pm^{\alpha_\pm}.
\]
Here \(\rho^{(i)}\) denotes the L\'evy measure defined in \eqref{def:measureRho}.
The origin component is obtained directly from Proposition \ref{prop:Gamma0LawAndIndependence}:
\[
\bE[e^{-q\Gamma_{1}^{(0)}}] = e^{-cb_0q^{\gamma_0}} \quad(q\ge 0).
\]
Consequently,
\[
\rho^{(0)}((t,\infty)) = a_0^\ast t^{-\alpha_0^\ast}
\quad(t>0),
\quad a_0^\ast:=\frac{cb_0}{\Gamma(1-\gamma_0)},
\quad
\alpha_0^\ast:=\gamma_0.
\]
If \(c=0\), then \(\Gamma^{(0)}\equiv0\), and Corollary \ref{cor:zeroStagnancyPreservation} gives \(O^{\ast,(0)}\equiv0\).
If \(c>0\), then \(\Gamma^{(0)}\) is a nonzero \(\gamma_0\)-stable subordinator.

To determine the dominant components, define
\[
\cA:=
\begin{cases}
\{+,-\} &(c=0),\\
\{+,-,0\} &(c>0).
\end{cases}
\]
Thus, \(\cA\) is the set of non-zero components.
Set
\[
\alpha:=\min_{i\in\cA}\alpha_i^\ast,
\quad
I_{0,\mathrm{dom}}
:=
\{i\in\cA\mid\alpha_i^\ast=\alpha\},
\quad
K_\ast:=\sum_{j\in I_{0,\mathrm{dom}}}a_j^\ast.
\]
Taking the common slowly varying function to be
\(K(t)\equiv K_\ast\), we obtain
\[
\frac{\rho^{(i)}((t,\infty))}
     {t^{-\alpha}K_\ast}
\longrightarrow\beta_i
\quad(t\to\infty,\ i\in\{+,-,0\}),
\]
where
\[
\beta_i:=
\begin{cases}
a_i^\ast/K_\ast &(i\in I_{0,\mathrm{dom}}),\\
0 &(i\notin I_{0,\mathrm{dom}}).
\end{cases}
\]
In particular, \(\sum_{i\in\{+,-,0\}}\beta_i=1\), and the dominant
regular variation condition of
Theorem~\ref{thm:occupation_limit} holds with
\[
h(x)=\frac{x^\alpha}{\Gamma(1-\alpha)K_\ast}.
\]

Every non-dominant component in \(\cA\) is sub-dominant: it has a positive lifetime-tail coefficient and an index \(\alpha_i^\ast\in(\alpha,1)\).
Hence,
\[
I_{0,\mathrm{sub}}
=
\cA \setminus I_{0,\mathrm{dom}},
\quad
I_{0,\mathrm{res}}
=
\begin{cases}
\{0\} &(c=0),\\
\emptyset &(c>0).
\end{cases}
\]
For \(c>0\), suitable choices of
\(\gamma_+,\gamma_-,\gamma_0\) can therefore make any nonempty
subset of \(\{+,-,0\}\) the dominant set.
For \(c=0\), only the two excursion components enter this comparison.

This example also illustrates the different descriptions of the limiting zero set \(\cZ\) in Theorem \ref{thm:waiting-times}.
If \(c=0\), or if \(c>0\) and
\[
\gamma_0>\min\{\alpha_+\gamma_+,\alpha_-\gamma_-\},
\]
then \(\beta_0=0\), and \(\cZ\) is the closed range
\(\cR_S\) of the stable subordinator \(S\).

If \(c>0\) and
\[
\gamma_0=\min\{\alpha_+\gamma_+,\alpha_-\gamma_-\},
\]
then \(0<\beta_0<1\).
The set \(\cZ\) is obtained by filling in the gaps of \(\cR_S\) associated with the origin component \(S^{(0)}\).

Finally, if \(c>0\) and
\[
\gamma_0<\min\{\alpha_+\gamma_+,\alpha_-\gamma_-\},
\]
then \(\beta_0=1\).
Thus, \(\cZ=[0,\infty)\) and both limiting waiting times are the identity function.

\appendix

\section{Proofs of auxiliary results} \label{appendix:proofs}

\begin{proof}[Proof of Lemma \ref{lem:dense_pointwise_inverse}]
Fix either a continuity point \(t>0\) of \(f^{-1}\), or \(t=0\)
when \(f^{-1}(0)=0\), and set
\[
s:=f^{-1}(t).
\]
Let \(\eps>0\).
Choose \(b\in D\cap(s,s+\eps)\).
Since \(b>s\), we have \(f(b)>t\), and hence \(f_n(b)>t\) for
all sufficiently large \(n\).
Thus,
\[
\limsup_{n\to\infty}f_n^{-1}(t)\le b<s+\eps.
\]
If \(s=0\), this already proves the desired convergence.

Suppose now that \(s>0\).
Then by assumption, \(t>0\).
Choose \(a\in D\cap((s-\eps)\vee0,s)\).
By the definition of \(s\), we have \(f(a)\le t\).
If \(f(a)=t\), then \(f^{-1}(v)\le a\) for every \(0\le v<t\),
so
\[
f^{-1}(t-)\le a<s=f^{-1}(t),
\]
contrary to the continuity of \(f^{-1}\) at \(t\).
Therefore, \(f(a)<t\), and \(f_n(a)<t\) for all sufficiently
large \(n\).
Consequently,
\[
s-\eps<a
\le\liminf_{n\to\infty}f_n^{-1}(t)
\le\limsup_{n\to\infty}f_n^{-1}(t)
<s+\eps.
\]
Letting \(\eps\downarrow0\) proves the pointwise convergence.

Finally, if \(f\) is strictly increasing, then \(f^{-1}\) is
continuous and \(f^{-1}(0)=0\).
Hence, \(f_n^{-1}(t)\to f^{-1}(t)\) for every \(t\ge0\).
Since the functions \(f_n^{-1}\) are non-decreasing and the limit
is continuous, the convergence is locally uniform.
\end{proof}

\begin{proof}[Proof of Lemma \ref{lem:inverse_composition_pointwise}]
  From Lemma \ref{lem:dense_pointwise_inverse}, the functions \(y_{n}^{-1}\) converge to \(y^{-1}\) pointwise.
  Fix \(t \geq 0\) such that \(x\) is continuous at \(y^{-1}(t)\), and set \( s := y^{-1}(t)\).
  Let \(T > s\).
  Since \(x_n \to x\) in \(D([0,\infty),\bR)\), there exist \(\lambda_n \in \Lambda\) such that
  \[
  \sup_{u\in[0,T]}
  \bigl(
    |x_n(\lambda_n(u))-x(u)|
    \vee
    |\lambda_n(u)-u|
  \bigr)
  \to 0
  \quad (n\to\infty).
  \]
  Setting
  \[
  s_{n} := y_n^{-1}(t),
  \quad
  r_{n} := \lambda_n^{-1}(s_n),
  \]
  we readily have \(s_n\to s\) and \(r_n\to s\).
  Then 
  \begin{align}
  |x_n(s_{n})-x(s)|
  &\leq
  |x_n(\lambda_n(r_{n}))-x(r_{n})|
  +
  |x(r_{n})-x(s)|.
  \end{align}
  Since the first term is bounded by \(\sup_{u\in[0,T]}|x_n(\lambda_n(u))-x(u)|\) and the second term also converges to \(0\),
  we obtain \(x_n(y_n^{-1}(t))\to x(y^{-1}(t))\ (n\to\infty)\).
  The proof is complete.
\end{proof}

\begin{proof}[Proof of Lemma \ref{lem:coordinatewise_to_multivariate_J1}]
We first consider the case \(d=2\), and write
\[
\alpha_n:=x_n^{(1)},\quad \alpha:=x^{(1)},\quad
\beta_n:=x_n^{(2)},\quad \beta:=x^{(2)}.
\]
By assumption, \(\alpha_n\to\alpha\) in \(D([0,\infty),\bR)\) and \(\beta_n\to\beta\) in \(D([0,\infty),\bR)\).

We claim that, for every \(t\ge0\), there exists a sequence \(t_n\to t\) such that
\[
\Delta\alpha_n(t_n)\to\Delta\alpha(t),
\quad
\Delta\beta_n(t_n)\to\Delta\beta(t).
\]
Indeed, suppose first that \(\Delta\alpha(t)\neq0\).
Then, by \cite[Proposition VI.2.1(a)]{JacodShiryaev}, there exists \(t_n\to t\) such that
\[
\Delta\alpha_n(t_n)\to\Delta\alpha(t).
\]
Since \(\alpha\) and \(\beta\) have no common jump times, we have \(\Delta\beta(t)=0\), that is, \(\beta\) is continuous at \(t\).

We show that \(\Delta\beta_n(t_n)\to0\).
Since \(\beta_n\to\beta\) in \(D([0,\infty),\bR)\) with the \(J_1\)-topology, there exist strictly increasing continuous function \(\lambda_n\) on \([0,\infty)\) such that \(\lambda_{n}(0) = 0\), \(\lambda_{n}(\infty) = \infty\), and for each \(T > 0\)
\[
\sup_{s\in[0,T]}|\lambda_n(s)-s|\to0,
\quad
\sup_{s\in[0,T]}|\beta_n(\lambda_n(s))-\beta(s)|\to0.
\]
Set \(s_n:=\lambda_n^{-1}(t_n)\). Then \(s_n\to t\).
Since \(\beta_n(t_n)=\beta_n(\lambda_n(s_n))\),
it follows that \(\beta_n(t_n)\to\beta(t)\).
Moreover, taking any \(T > t\), we have
\[
|\beta_{n}(t_{n}-)-\beta(s_{n}-)| = |\beta_{n}(\lambda_{n}(s_{n})-)-\beta(s_{n}-)| 
\le
\sup_{u\in[0,T]}|\beta_n(\lambda_{n}(u))-\beta(u)|
\to0,
\]
where we note that \(\lim_{\eps \downarrow 0 }\beta_{n}(\lambda_{n}(s_{n}-\eps)) = \beta_{n}(\lambda_{n}(s_{n})-)\).
Since \(\beta\) is continuous at \(t\), we have \(\beta_{n}(t_{n}-) \to \beta(t)\), and we obtain \(\Delta\beta_n(t_n) \to 0\).

The case \(\Delta\beta(t)\neq0\) is symmetric.
Finally, if \(\Delta\alpha(t)=\Delta\beta(t)=0\), then taking \(t_n=t\) and applying the same argument to both coordinates, we obtain
\[
\Delta\alpha_n(t)\to0,
\quad
\Delta\beta_n(t)\to0.
\]
Thus, the claim is proven.

Then, by \cite[Proposition VI.2.2(b)]{JacodShiryaev}, we have the joint convergence
\[
(\alpha_n,\beta_n)\to(\alpha,\beta)
\quad
\text{in } D([0,\infty),\bR^2).
\]

The general case follows by induction on \(d\).
\end{proof}

\section{Non-Markovianity of inverse-subordinator time changed processes}
\label{appendix:nonMarkov}

We prove the assertion in Remark \ref{rem:non-preservation}.
Our purpose in this appendix is to show the following proposition.
Note that this result does not require the original process to be regenerative or Markovian.

\begin{Prop}\label{prop:inverse-clock-nonMarkov}
Let \(Z\) be a subordinator without killing, with drift coefficient
\(d\ge0\) and L\'evy measure \(\nu\) satisfying
\(\nu((0,\infty))=\infty\).
Let \(X\) be an independent c\`adl\`ag process taking values in a
Polish space \(S\), and assume that, almost surely, its path is not
constant on any non-trivial time interval:
\begin{equation}\label{eq:nonMarkov-no-flat}
\bP\left[ X_{u}=X_{s}\text{ for every }u\in[s,t] \right]=0 \quad \text{for every \(0 \leq s < t\)}.
\end{equation}
Define
\[
A_t:=\inf\{s\ge0:Z_s>t\},
\quad
X^\ast_t:=X_{A_t}
\quad(t\ge0).
\]
Then \(X^\ast\) does not satisfy the time-inhomogeneous Markov property with respect to its completed natural filtration \((\cG_{t})_{t \geq 0}\).
More precisely, for every \(t > 0\) there exists a set \(E_{t}\) measurable with respect to \(\sigma(X^{\ast}_{u},u \geq t)\) and
\[
\bP[E_{t} \mid \cG_{t}] \neq \bP[E_{t} \mid X^{\ast}_{t}]
\]
with positive probability.
\end{Prop}

For the proof, we use two elementary facts about subordinators.
Although they are classical we give a short proof for completeness
(see also \cite[Theorem 24.10]{SatoLevy} for a general support result related to the first lemma).

Throughout the remainder of this appendix, \(Z,A,d,\nu\) are as in
Proposition \ref{prop:inverse-clock-nonMarkov}, and we write
\[
\mu_s:=\cL(Z_s),
\quad
\overline\nu(x):=\nu((x,\infty))
\quad(s\ge0,\ x>0).
\]
Let \(N(dr\,dx)\) denote the Poisson point process of the jumps
of \(Z\), with intensity \(dr\,\nu(dx)\).
We state the lemmas first and defer their proofs until after the
proof of the proposition.

\begin{Lem}\label{lem:suppOfSubordinator}
Set \(\mu_s:=\cL(Z_s)\).
For \(s > 0\), we have
\begin{equation}\label{eq:nonMarkov-support}
\Supp \mu_s=[sd,\infty)
\quad(s>0),
\end{equation}
where \(\Supp \mu_{s}\) denotes the topological support of \(\mu_{s}\).
\end{Lem}

\begin{Lem} \label{lem:overAndUnderShoots}
For \(t>0\), define the undershoot, overshoot, and jump-crossing
event by
\[
U_t:=t-Z_{A_t-},
\quad
O_t:=Z_{A_t}-t,
\quad
J_t:=\{Z_{A_t-}<t<Z_{A_t}\}.
\]
Then we have
\begin{equation}\label{eq:nonMarkov-crossing-events}
\bP[J_{t}] > 0, \quad J_t=\{U_t>0\}=\{O_t>0\}
\quad\text{a.s.},
\end{equation}
and
\begin{align}
\bP[A_t\in ds,U_t\in du,O_t \in dv,J_t] = 1\{0 < u < t,v > 0\}\,ds \, \mu_{s}(t-du)\,\nu(u + dv).
\label{eq:nonMarkov-age-residual-law}
\end{align}
Moreover, for every \(h>0\),
\[
\bP[U_{t}=h] = \bP[O_{t} = h]=0.
\]
\end{Lem}

\begin{proof}[Proof of Proposition \ref{prop:inverse-clock-nonMarkov}]
  
Since \(\nu((0,\infty))=\infty\), the paths of \(Z\) are strictly
increasing and its inverse \(A\) is continuous.

Fix an arbitrary \(t>0\), and choose \(h>0\) such that
\(\overline\nu(2h)>0\).
Set
\[
E_{t,h}:=
\{X^\ast_u=X^\ast_t\text{ for every }u\in[t,t+h]\}.
\]
This event belongs to \(\sigma(X^{\ast}_{u},u\geq t)\).
By the continuity of \(A\) and the absence of non-trivial constant
intervals of \(X\), an interval is a constant interval of \(X^\ast\)
if and only if it is a constant interval of \(A\).
Hence,
\[
U_t = \sup\left\{ a\in[0,t] \mid X^\ast_u=X^\ast_t\text{ for every }u\in[t-a,t]
\right\}
\quad\text{a.s.}
\]
Thus, \(U_t\) is \(\cG_t\)-measurable.
Similarly, using Lemma \ref{lem:overAndUnderShoots},
\begin{equation}\label{eq:nonMarkov-plateau-event}
E_{t,h}=\{O_t\ge h\}=\{O_t>h\}\quad\text{a.s.}
\end{equation}

Suppose, for a contradiction, that 
\[
\bP[E_{t,h} \mid \cG_{t}] = \bP[E_{t,h} \mid X^{\ast}_{t}] \quad \text{a.s.}
\]
Then there exists a measurable function \(q_{t,h}:S\to[0,1]\) such that
\begin{equation}\label{eq:nonMarkov-Markov-property}
\bP[E_{t,h}\mid \cG_t]=q_{t,h}(X^\ast_t)
\quad\text{a.s.}
\end{equation}
Conditioning \eqref{eq:nonMarkov-Markov-property} on
\((U_t,X^\ast_t)\), we obtain
\begin{equation}\label{eq:nonMarkov-state-prediction}
\bP[E_{t,h}\mid U_t,X^\ast_{t}]=q_{t,h}(X^\ast_t)
\quad\text{a.s.}
\end{equation}

On the other hand, Lemma \ref{lem:overAndUnderShoots} gives, for
\(a \geq 0\),
\begin{align}
\bP[A_t\in ds,U_t\in du,O_t>a,J_t] = 
\overline\nu(u+a) 1\{0 < u < t\}\,ds\,
\,\mu_s(t-du). \label{eq:nonMarkov-age-tail-law}
\end{align}
Comparing this expression at
\(a=0\) and \(a=h\) yields
\[
\bP[O_{t}>h\mid A_t,U_t]
=
\frac{\overline\nu(U_t+h)}{\overline\nu(U_t)}
\quad\text{a.s. on }J_t.
\]
Define
\[
k_h(u):=
\begin{cases}
\overline\nu(u+h)/\overline\nu(u)
&(u>0,\ \overline\nu(u)>0),\\
0,&\text{otherwise}.
\end{cases}
\]
Using \eqref{eq:nonMarkov-crossing-events} and
\eqref{eq:nonMarkov-plateau-event}, we may therefore write
\begin{equation}\label{eq:nonMarkov-clock-prediction}
\bP[E_{t,h}\mid A_{t},U_{t}]=k_h(U_t)
\quad\text{a.s.}
\end{equation}

Since \(X^\ast_t=X_{A_t}\) is conditionally independent of \((U_t,O_t)\)
given \(A_t\) by the independence of \(X\) and \(Z\), we have
\[
\bP[E_{t,h}\mid A_{t},U_{t},X^\ast_{t}]=k_h(U_t)
\quad\text{a.s.}
\]
Taking conditional expectations with respect to
\((U_t,X^\ast_t)\) yields
\[
\bP[E_{t,h}\mid U_{t},X^{\ast}_{t}]=k_h(U_t)
\quad\text{a.s.}
\]
Together with \eqref{eq:nonMarkov-state-prediction}, this gives
\begin{equation}\label{eq:nonMarkov-equal-predictions}
k_h(U_t)=q_{t,h}(X^\ast_t)\quad\text{a.s.}
\end{equation}
Finally, condition both sides on \((A_t,U_t)\). The left-hand side
is unchanged, whereas independence of \(X\) and \(Z\) gives
\[
\bE[q_{t,h}(X^\ast_t)\mid A_t,U_t]
=
 r_{t,h}(A_t) \quad \text{a.s.} 
\quad
\text{for} \quad r_{t,h}(s):=\bE[q_{t,h}(X_s)].
\]
Thus, it follows
\begin{equation}\label{eq:nonMarkov-degeneracy}
k_h(U_t)=r_{t,h}(A_t)\quad\text{a.s.}
\end{equation}
This implies that \(k_h(U_t)\) is almost surely constant given \(A_{t}\).

Restricting \eqref{eq:nonMarkov-degeneracy} to \(J_t\) and using \eqref{eq:nonMarkov-age-residual-law}, we have
\begin{align}
0
&=
\bE\left[\bigl(k_h(U_t)-r_{t,h}(A_t)\bigr)^2,J_t\right]
\nonumber\\
&=
\int_0^\infty ds\int_{(0,t)}
\bigl(k_h(u)-r_{t,h}(s)\bigr)^2
\overline\nu(u)\,\mu_s(t-du).
\label{eq:nonMarkov-square-identity}
\end{align}
Fix \(s_{0}>0\) such that \(s_{0}d<t\).
To obtain a contradiction, it suffices to show that
\begin{align}  
\int_{(0,t)}
\bigl(k_h(u)-r_{t,h}(s)\bigr)^2
\overline\nu(u)\,\mu_s(t-du) > 0 \quad \text{for every}\quad s \in (0,s_{0}). \label{eq:non-constant}
\end{align}

Put
\[
b_{s_{0}}:=h\wedge(t-s_{0}d)>0.
\]
By Lemma \ref{lem:suppOfSubordinator}, for \(s \in (0,s_{0})\), the measure
\[
1\{0 < u < t\}\overline\nu(u)\,\mu_{s}(t-du)
\]
assigns positive mass to every non-empty open subinterval of
\((0,b_{s_{0}})\).
Moreover,
\[
k_h(u)>0\quad(0<u<b_{s_{0}}),
\quad
0\le k_h(u)
\leq
\frac{\overline\nu(h)}{\overline\nu(u)}
\longrightarrow0
\quad(u\downarrow0),
\]
where positivity follows from
\(\overline\nu(u+h)\ge\overline\nu(2h)>0\) for \(u<h\),
and the limit follows from \(\nu((0,\infty))=\infty\).
This implies that \(k_{h}(u)\) takes arbitrarily positive small value as \(u \to 0\).
Then \eqref{eq:non-constant} follows, and this is a contradiction.

We have proved that
\[
\bP\left[
\bP[E_{t,h}\mid\cG_t]
\ne
\bP[E_{t,h}\mid X^\ast_t]
\right]>0.
\]
Since \(t>0\) was arbitrary, taking \(E_t=E_{t,h}\) proves the
desired result.
\end{proof}

We now give the deferred proofs.

\begin{proof}[Proof of Lemma \ref{lem:suppOfSubordinator}]
It suffices to consider \(d=0\) since the general case follows
by adding the deterministic drift \(sd\).
Fix \(s>0\) and \(0\le a<b\).
We show \(\bP[Z_{s} \in (a,b)] > 0\).
Since \(\inf \Supp\nu = 0\), we can choose \(z\in \Supp\nu\)
and an integer \(k\ge1\) such that \(a<kz<b\).
Choose \(\delta \in (0,z)\) sufficiently small so that
\[
a<k(z-\delta)<k(z+\delta)<b.
\]
Put \(B := (z-\delta,z+\delta)\) and \(\rho := b-k(z+\delta)>0\).
Since
\[
\bE\left[\int_{[0,s]\times (0,\eps]}x N(dr \,dx)\right]
= s\int_{(0,\eps]}x \nu(dx) \longrightarrow 0 \quad(\eps\downarrow0),
\]
we can choose \(\eps \in (0,z-\delta)\) such that
\[
\bP\left[\int_{[0,s]\times (0,\eps]}x N(dr \,dx) < \rho \right] > 0.
\]
Then by the independence property of Poisson point processes we have
\begin{align}  
&\bP\left[ \int_{[0,s]\times (0,\eps]}x N(dr \,dx) < \rho, N([0,s] \times B) = k, N([0,s] \times ((\eps,\infty) \setminus B)) = 0 \right] \\
= 
&\bP\left[ \int_{[0,s]\times (0,\eps]}x N(dr \,dx) < \rho\right] \bP[N([0,s] \times B) = k]\bP[N([0,s] \times ((\eps,\infty) \setminus B)) = 0]\\
= &\bP\left[ \int_{[0,s]\times (0,\eps]}x N(dr \,dx) < \rho\right]\mathrm{e}^{-s\nu(B)}\frac{(s\nu(B))^{k}}{k!}e^{-s\nu((\eps,\infty) \setminus B)} > 0.
\end{align}
Since \(Z_{s} \in (a,b)\) on the event in the leftmost side, we obtain \eqref{eq:nonMarkov-support}.
\end{proof}

\begin{proof}[Proof of Lemma \ref{lem:overAndUnderShoots}]
Set the renewal measure by
\[
W(B):=\int_0^\infty\mu_s(B) ds.
\]
The measure \(W\) is diffuse since \(Z\) is strictly increasing almost surely.
For \(t>0\) and \(\eps>0\), the compensation formula (see e.g., \cite[Theorem 4.4]{KyprianouText}) gives
\begin{align*}
\bE \left[
\sum_{s>0}1\{\Delta Z_{s}>\eps\}
\bigl(1\{Z_{s-}=t\}+1\{Z_{s}=t\}\bigr)
\right]=
\overline\nu(\eps)W(\{t\}) + \int_{(\eps,\infty)}W(\{t-x\})\nu(dx)
=0.
\end{align*}
Letting \(\eps\downarrow0\), we see that
\begin{equation}
J_t=\{U_t>0\}=\{O_t>0\}
\quad\text{a.s.}
\end{equation}
We also see that a fixed positive level is almost surely not an endpoint of
a jump interval.
In particular, for every \(h>0\),
\[
\bP[U_{t}=h] = \bP[O_{t} = h]=0.
\]

We next identify the joint law at the crossing of level \(t\).
For every non-negative measurable function \(F\), the compensation
formula yields
\begin{align*}
\bE[F(A_t,Z_{A_t-},Z_{A_t}),J_t] &=
\bE\left[
\sum_{s>0}F(s,Z_{s-},Z_s)
1_{\{Z_{s-}<t<Z_s\}}
\right] \\
&=
\int_0^\infty ds\int_{[0,t)}\mu_s(du)
\int_{(t-u,\infty)}F(s,u,u+z)\,\nu(dz).
\end{align*}
Here we used that \(Z_{s-}\) and \(Z_s\) have the same law for each
fixed \(s\).
Equivalently,
\begin{align}
\bP[A_{t}\in ds,Z_{A_{t}-}\in du,Z_{A_{t}}\in dv,J_{t}] =
1\{0\le u<t<v\}\,
 ds\,\mu_s(du)\,\nu(-u+dv).
\label{eq:nonMarkov-crossing-law}
\end{align}
Applying the map \((s,u,v)\mapsto(s,t-u,v-t)\), and using
\(\mu_s(\{0\})=0\) for \(s>0\), gives
\eqref{eq:nonMarkov-age-residual-law}.
In particular,
\begin{equation}\label{eq:nonMarkov-jump-probability}
\bP[J_{t}] = \int_0^\infty ds\int_{[0,t)}
\overline\nu(t-u)\,\mu_s(du).
\end{equation}
This probability is positive for every \(t>0\).
To see this, choose \(0<\eps<t\) with
\(\overline\nu(\eps)>0\), and then choose \(s_{0}>0\) such
that \(s_{0} d<t-\eps\).
By Lemma \ref{lem:suppOfSubordinator}, we have
\[
\bP[J_{t}]
\geq
\overline\nu(\eps)
\int_0^{s_0}\mu_s((t-\eps,t))\,ds
>0.
\]
The proof is complete.
\end{proof}

\bibliographystyle{plain}

\end{document}

%% file: style-common.tex
\makeatletter
\@addtoreset{footnote}{page}
\makeatother

\usepackage{enumitem}

\usepackage{amsthm}
\usepackage{amsmath,amssymb,latexsym,amsfonts,mathrsfs}

\newcommand{\Supp}{\mathop{\mathrm{Supp}}}

\newcommand{\eps}{\ensuremath{\varepsilon}}

\renewcommand{\tilde}{\widetilde}

\newcommand{\bE}{\ensuremath{\mathbb{E}}}

\newcommand{\bN}{\ensuremath{\mathbb{N}}}

\newcommand{\bP}{\ensuremath{\mathbb{P}}}

\newcommand{\bR}{\ensuremath{\mathbb{R}}}

\newcommand{\bZ}{\ensuremath{\mathbb{Z}}}

\newcommand{\cA}{\ensuremath{\mathcal{A}}}
\newcommand{\cB}{\ensuremath{\mathcal{B}}}

\newcommand{\cE}{\ensuremath{\mathcal{E}}}
\newcommand{\cF}{\ensuremath{\mathcal{F}}}
\newcommand{\cG}{\ensuremath{\mathcal{G}}}
\newcommand{\cH}{\ensuremath{\mathcal{H}}}

\newcommand{\cL}{\ensuremath{\mathcal{L}}}
\newcommand{\cM}{\ensuremath{\mathcal{M}}}

\newcommand{\cR}{\ensuremath{\mathcal{R}}}

\newcommand{\cT}{\ensuremath{\mathcal{T}}}

\newcommand{\cZ}{\ensuremath{\mathcal{Z}}}

%% file: style-e.tex
\theoremstyle{plain}
\newtheorem{Thm}{Theorem}[section]

\newtheorem{Lem}[Thm]{Lemma}
\newtheorem{Prop}[Thm]{Proposition}
\newtheorem{Cor}[Thm]{Corollary}

\theoremstyle{definition}

\newtheorem{Rem}[Thm]{Remark}
\newtheorem{Ex}[Thm]{Example}

%% file: style-layout-paper.tex
\numberwithin{equation}{section}

\makeatletter
\renewcommand\section{\@startsection {section}{1}{\z@}%
                                   {-3.5ex \@plus -1ex \@minus -.2ex}%
                                   {2.3ex \@plus.2ex}%
                                   {\normalfont\large\textbf}}
\makeatother

\makeatletter
\renewcommand\subsection{\@startsection {subsection}{1}{\z@}%
                                   {-3.5ex \@plus -1ex \@minus -.2ex}%
                                   {2.3ex \@plus.2ex}%
                                   {\normalfont\normalsize\textbf}}
\makeatother